\documentclass[10.5pt]{article}

\usepackage[dvipdfmx]{graphicx}
\usepackage{amsmath}
\usepackage{amscd}
\usepackage{amssymb}
\usepackage{amsthm}
\usepackage{bm}
\usepackage[all]{xy}
\usepackage{color}
\usepackage{url}
 
\usepackage{algorithmic}
\usepackage{algorithm}

\theoremstyle{definition}
\newtheorem{theorem}{Theorem}[section]

\newtheorem{proposition}[theorem]{Proposition}
\newtheorem{corollary}[theorem]{Corollary}
\newtheorem{example}[theorem]{Example}

\makeatletter 
\@tempcnta\z@
\loop\ifnum\@tempcnta<26
\advance\@tempcnta\@ne
\expandafter\edef\csname f\@Alph\@tempcnta\endcsname{\noexpand\mathfrak{\@Alph\@tempcnta}}
\expandafter\edef\csname l\@Alph\@tempcnta\endcsname{\noexpand\mathbb{\@Alph\@tempcnta}}
\expandafter\edef\csname c\@Alph\@tempcnta\endcsname{\noexpand\mathcal{\@Alph\@tempcnta}}
\expandafter\edef\csname b\@Alph\@tempcnta\endcsname{\noexpand\mathbf{\@Alph\@tempcnta}}
\repeat

\def\Inf{\operatornamewithlimits{inf\vphantom{p}}}

\newcommand{\dd}{{\delta}}
\newcommand{\DD}{{\Delta}}
\newcommand{\ee}{{\varepsilon}}

\newcommand{\convd}{\to_{\mathrm{d}}}

\newcommand{\convas}{\xrightarrow{\mathrm{a.s.}}}

\newcommand{\iid}{{\mathrm{i.i.d.}}}
\newcommand{\prob}{{\mathrm{Pr}}}

\newcommand{\median}{\mathop{\mathrm{median}}}

\newcommand{\one}{\mathrm{I}}
\newcommand{\field}{\mathrm{\bf k}}

\newcommand{\Filt}{\mathrm{{\lF}ilt}}

\newcommand{\Lip}{\mathop{\mathrm{Lip}}}

\newcommand{\dH}{d_{\mathrm{H}}}

\providecommand{\dW}[1]{d_{\mathrm{W}_{#1}}}

\newcommand{\pers}{\mathrm{pers}}

\newcommand{\ad}{{\lR^{2}_{\mathrm{ad}}}}

\newcommand{\Span}{\mathop{\mathrm{Span}}}
\newcommand{\kG}{k_{\mathrm{G}}}

\newcommand{\Kc}{\cK_{\mathrm{c}}}
\newcommand{\Bdd}{\mathop{\mathrm{Bdd}}}

\providecommand{\abs}[1]{\left\lvert#1\right\rvert}  
\providecommand{\norm}[1]{\left\lVert#1\right\rVert}  
\providecommand{\pare}[1]{\left( #1 \right)}  
\providecommand{\braces}[1]{\left\{ #1 \right\}}  
\providecommand{\bracket}[1]{\left[ #1 \right]}  
\newcommand{\remiddle}{\mathrel{}\middle|\mathrel{}} 
\providecommand{\card}[1]{\mathrm{card}\pare{#1}}  
\providecommand{\inn}[2]{\langle #1, #2 \rangle}  

\title{On the Expectation of a Persistence Diagram \\ by the Persistence Weighted Kernel}
\author{Genki Kusano\thanks{Tohoku University (genki.kusano.r5@dc.tohoku.ac.jp).}}
\date{}
\begin{document}
\maketitle

%
%
%
%
%
%
%
%
%
\begin{abstract}
In topological data analysis, persistent homology characterizes robust topological features in data and it has a summary representation, called a persistence diagram.
Statistical research for persistence diagrams have been actively developed, and the persistence weighted kernel shows several advantages over other statistical methods for persistence diagrams.
If data is drawn from some probability distribution, the corresponding persistence diagram have randomness.
Then, the expectation of the persistence diagram by the persistence weighted kernel is well-defined.
In this paper, we study relationships between a probability distribution and the persistence weighted kernel in the viewpoint of 
(1) the strong law of large numbers and the central limit theorem, 
(2) a confidence interval to estimate the expectation of the persistence weighted kernel numerically, and 
(3) the stability theorem to ensure the continuity of the map from a probability distribution to the expectation.
In numerical experiments, we demonstrate our method gives an interesting counterexample to a common view in topological data analysis.
\end{abstract}

\section{Introduction}
{\em Topological data analysis} ({\em TDA}) is a research field utilizing topological properties in data analysis, and there are a lot of successful works in various research fields; for example, material science \cite{NHHEN15,HNHEMN16,STRFH17}, biochemistry \cite{CMWOXW15,GHIKMN13}, information science \cite{HK16,dSG07}, fluid dynamics \cite{KLTSXPSM16}, computer vision \cite{SOCG10}, and so on.
One of the key tools in TDA is {\em persistent homology}, which was first introduced in \cite{ELZ02} to describe robust topological features in data.

For explaining persistent homology intuitively, let us consider that input data is given as a point set.
This situation appears in many applications in TDA.
For example, by representing each atom in some material as its $3$-dimensional coordinates, the atomic configuration of the material is represented as a point set of $\lR^{3}$.
As another example, let $\mu$ be a probability measure on a measurable space $M$, then observations $\bm{x}_{1},\ldots,\bm{x}_{n}$ drawn from $\mu$ are also expressed as a point set $\{\bm{x}_{1},\ldots,\bm{x}_{n}\} \subset M$.
For a point set $X \subset \lR^{d}$, we consider an union of balls $B(X;a):=\bigcup_{\bm{x} \in X} B(\bm{x};a)$ with radius $a>0$, where $B(\bm{x};a)=\{\bm{z} \in \lR^{d} \mid \norm{\bm{x}-\bm{z}} \leq a\}$, and the homology group $H_{q}(B(X;a))$ in dimension $q$ with field coefficient.
Let $u^{b}_{a}:H_{q}(B(X;a)) \to H_{q}(B(X;b)) ~ (a \leq b) $ denote the induced map of the inclusion $X_{a} \hookrightarrow X_{b}$.
Here, we view a nonzero homology class $\alpha \in H_{q}(X;a)$ representing a $q$-dimensional topological hole in $B(X;a)$ and call it a {\em generator} of the persistent homology.
If $u^{b}_{a}(\alpha) \in H_{q}(B(X;b))$ is not zero, then it is interpreted that a generator $\alpha \in H_{q}(X;a)$ exists at radius $a$ and still {\em persists} at radius $b$. 
Here, the maximum and minimum of a radius $a$ at which a generator $\alpha \in H_{q}(B(X;a))$ persists are called the {\em birth time} and the {\em death time}, respectively.
The collection of all birth-death pairs of generators is a multiset of $\lR^{2}$ and called the {\em persistence diagram} (Figure \ref{fig:ball_pd}).
For a birth-death pair $x=(b,d)$ in a persistence diagram, its lifetime $d-b$ is called the {\em persistence} of $x$.
As Figure \ref{fig:ball_pd}, a generator with large (resp. small) persistence is interpreted to represent a large (resp. small) topological hole.
Hence, in many applications, generators with large persistence are viewed as important features in the input data, and generators with small persistence are treated as topological noise.
\begin{figure}[htbp]
\begin{center}
\caption{Unions of balls with several radii (left) and the persistence diagram in dimension $1$ (right). The persistence diagram encodes the birth and death time of each generator.}
\label{fig:ball_pd}
\includegraphics[width=1\textwidth]{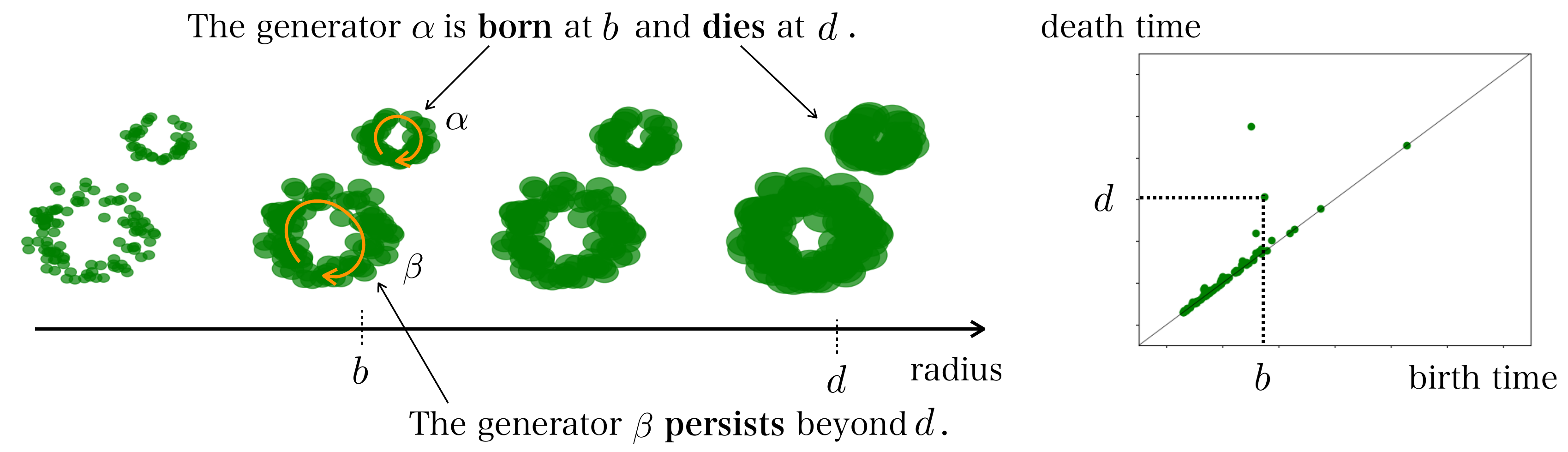}
\end{center}
\end{figure}

In applying a persistence diagram to various problems, it is desirable to analyze persistence diagrams statistically.
However, the definition of a persistence diagram as a multiset of $\lR^{2}$ is inappropriate for directly considering its statistical properties.
While standard statistical methods assume that the input data lies in a space with an inner product structure, a natural inner product for the space of persistence diagrams is not proposed.
This gap is an obstacle to developing statistical methods for persistence diagrams.
In order to avoid this obstacle, many research in statistical TDA consider to transform a persistence diagram to an element in a Hilbert space \cite{AEKNPSCHMZ17,Bu15,CMWOXW15,CCO17,DFL98,KFH18,RHBK15,RT16}.
In this paper, we focus on the {\em persistence weighted kernel} ({\em PWK}) \cite{KFH16,KFH18}, \footnote{This was originally called the persistence weighted {\em Gaussian} kernel in \cite{KFH16,KFH18} because we mainly focused on the Gaussian kernel $k(x,y)=\exp(- \norm{x-y}^{2}/2\sigma^{2}) ~ (\sigma>0)$ as the positive definite kernel, but the framework can be generalized to other positive definite kernels. Hence, we drop the word ``Gaussian'' here.} which is a recently developed statistical method for persistence diagrams.
The PWK is composed of a {\em positive definite kernel} $k:\lR^{2} \times \lR^{2} \to \lR$ and a {\em weight function} $w:\lR^{2} \to \lR$.
While a one-variable function $k(\cdot,x):\lR^{2} \to \lR, ~ z \mapsto k(z,x) ~ (x \in \lR^{2})$ is in a functional space on $\lR^{2}$, it is shown that $k(\cdot,x)$ is an element in a Hilbert space, which is called the {\em reproducing kernel Hilbert space} (RKHS) of $k$.
Here, we transform a persistence diagram $D=\{x_{i}=(b_{i},d_{i}) \in \lR^{2} \mid i \in I\}$ to a weighted sum $V^{k,w}(D):=\sum_{x \in D} w(x)k(\cdot,x)$ and call it the {\em PWK vector}.

\subsection*{Contributions}
In an appropriate condition, a persistence diagram $D$ can be viewed as a sample drawn from some probability distribution $P$.
Then, the PWK vector $V^{k,w}(D)$ is a random variable taking values in the RKHS, and the expectation $\lE[V^{k,w}(D)]$ is well-defined.
Contributions in this paper are summarized as follows:

(1) Let $D_{1},\ldots,D_{n}$ be i.i.d. samples drawn from $P$ and $\overline{V}^{k,w}_{n}:=n^{-1}\sum_{i=1}^{n}V^{k,w}(D_{i})$ be the sample mean.
In order to understand probabilistic properties of the PWK vector, we will describe the convergence of $\overline{V}^{k,w}_{n}$ to $\lE[V^{k,w}(D)]$ by showing the strong law of large numbers (Theorem \ref{thm:SLNN_pwk}) and the central limit theorem for PWK vectors (Theorem \ref{thm:CLT_pwk}).

(2) From the strong law of large numbers for PWK vectors, we have the almost sure convergence $\overline{V}^{k,w}_{n} \convas \lE[V^{k,w}(D)]$ as $n$ goes to $\infty$.
In practice, the number of samples $n$ is finite.
Then, we are interested in how $\overline{V}^{k,w}_{n}$ for the fixed number $n$ is close to $\lE[V^{k,w}(D)]$.
We formalize this question in the context of a confidence interval.
A {\em confidence interval} for a parameter $\theta \in \lR$ is an interval which contains the parameter $\theta$ with high probability.
Since a PWK vector $V^{k,w}(D):\lR^{2} \to \lR, z \mapsto \sum_{x \in D}w(x)k(z,x)$ is a real-valued function on $\lR^{2}$, by regarding $\lE[V^{k,w}(D)]$ as a real-valued function on $\lR^{2}$, we construct the confidence interval for $\lE[V^{k,w}(D)]$ (Theorem \ref{thm:confi_interval_pwk}).

(3) When the sample mean $\overline{V}^{k,w}_{n}$ is viewed as the expectation of the empirical distribution of $D_{1},\ldots,D_{n}$, the difference between two expectations of a (true) distribution $P$ and its empirical distribution is estimated to be small from the strong law of large numbers.
Then, these probability distributions are also considered to be close with an appropriate distance on probability distributions when $n$ is large.
We will generalize this by showing that the map from a probability distribution of persistence diagrams to the expectation of a PWK vector is Lipschitz continuous (Theorem \ref{thm:stab_prob_pwk}), which is the stability of the expectation of a PWK vector.

\subsection*{Related work}
A first study of vectorizing a persistence diagram is a persistence landscape \cite{Bu15}, which is an $L^{p}(\lN \times \lR)$ function made from a persistence diagram, and \cite{Bu15} shows the strong law of large numbers and the central limit theorem for a persistence landscape.
With respect to other topics, there are previous works of a uniform confidence band for a persistence landscape \cite{CFLRSW13,CFLRW14} and the stability of the expectation of a persistence landscape \cite{CFLMRW15}.
We remark that we have used several techniques in these papers to show our results.
On the other hand, a PWK vector can flexibly measure a size of a generator by a weight function $w$, while several statistical methods for persistence diagrams, including persistence landscapes, always return a large (resp. small) value for a generator with large (resp. small) persistence.

In numerical experiments, we compare the PWK vector and other statistical methods for persistence diagrams: the persistence landscape, the persistence scale-space kernel \cite{Bu15}, and the Sliced Wasserstein kernel \cite{CCO17}.
While the persistence landscape and the Sliced Wasserstein kernel are parameter-free statistical methods, it is needed to select a positive definite kernel $k$ and a weight function $w$ for a PWK vector appropriately.
In other words, depending on datasets, we can select $k$ and $w$ flexibly to show higher performance of the PWK than other methods in statistical and machine learning tasks. 
One of the datasets which we will treat in this paper is related to material science, and it has important topological features whose persistence are small.
Utilizing the weight factor in a PWK vector, we will show the advantage of our statistical method.

\subsection*{Organization}
This paper is organized as follows: 
We review some the basics on persistence diagrams and a PWK vector in Section \ref{sec:preliminaries} and explain how to regard the PWK vector as a random variable taking values in the RKHS in Section \ref{sec:limit_theorems}.
Then, we show the limit theorems in Section \ref{sec:limit_theorems}, construct a confidence interval for the expectation of a PWK vector in Section \ref{sec:confidence}, and show the stability of the expectation of a PWK vector in Section \ref{sec:stability_expectation}.
In Section \ref{sec:numarical}, we show numerical results by using the PWK vector and compare performances with other statistical methods for persistence diagrams.

\section{Preliminaries}
\label{sec:preliminaries}

\subsection{Persistence Diagram}
\label{subsec:pd}
In this section, we briefly review the basics of persistent homology and persistence diagrams.
Throughout this section, we fix $\field$ a field.
We refer the reader to \cite{Ha02} for homology groups and \cite{ZC05} for persistence diagrams.

Let $\Filt=\{F_a\}_{a\in \lR}$ be a family of topological spaces.
If $F_{a}$ is a subspace of $F_{b}$ for any $a\leq b \in \lR$, $\Filt$ is called a {\em filtration} of topological spaces. 
Then, a family of the $q$-th homology groups $\{H_{q}(F_{a})\}_{a\in \lR}$ with coefficient $\field$ is a family of $\field$-vector spaces.
Since the inclusion $F_{a} \subset F_{b}$ $(a \leq b \in \lR)$ induces the $\field$-linear map $\phi^{b}_{a}:H_{q}(F_{a}) \to H_{q}(F_{b})$, there exists a family $\{ \phi^{b}_{a}\}_{a \leq b \in \lR}$ of $\field$-linear maps.
Then, the collection of $\{H_{q}(F_{a})\}_{a\in \lR}$ and $\{ \phi^{b}_{a}\}_{a \leq b \in \lR}$ is called the {\em $q$-th} {\em persistent homology} of $\Filt$ and denoted by $H_{q}(\Filt):=\{\{H_{q}(F_{a})\}_{a\in \lR}, \{\phi^{b}_{a}\}_{a\leq b \in \lR}\}$.

A persistent homology is generalized as follows:
Let $\{U_{a}\}_{a\in\lR}$ be a family of $\field$-vector spaces and $\{u^{b}_{a}:U_{a} \to U_{b}\}_{a\leq b \in \lR}$ be a family of $\field$-linear maps.
If $u^{a}_{a}$ is an identity map on $U_{a}$ and $u^{c}_{a}=u^{c}_{b} \circ u^{b}_{a}$ for any $a \leq b \leq c \in \lR$, the collection $\lU=\{\{U_{a}\}_{a\in \lR}, \{u^{b}_{a}\}_{a\leq b\in  \lR}\}$ of $\field$-vector spaces and $\field$-linear maps is called a {\em persistence module}.
For an interval\footnote{A subset $\bJ \subset \lR$ is said to be an {\em interval} if, for any $a,c \in \bJ$, $b \in \lR$ satisfying $a < b < c$ is in $\bJ$.} $\bJ \subset \lR$, the persistence module $\lI_{\bJ}=\{\{U_{a}\}_{a\in \lR}, \{u^{b}_{a}\}_{a\leq b \in \lR}\}$ is called the {\em interval persistence module} over $\bJ$ if $U_{a}=\field$ for $a \in \bJ$ and $U_{a}=0$ otherwise, and $u^{b}_{a}$ is an identity map for any $a \leq b \in \bJ$ and $u^{b}_{a}$ is a zero map otherwise.
It is known that an appropriate persistence module is decomposed to several interval persistence modules as $\lU \cong \bigoplus_{i \in L} \lI_{\bJ_{i}}$, where $L$ is an index set and each $\bJ_{i}$ is an interval of $\lR$.
We refer the reader to \cite{BS14,CdS10,CrB15,ZC05} for the details about the appropriate condition and the symbol $\cong$ and $\bigoplus$.
We remark that all persistence modules which will be used in this paper satisfies the decomposition condition.

For an interval $\bJ \subset \lR$, we define the {\em birth time} and {\em death time} of $\bJ$ by the endpoints $b(\bJ):=\inf \{b \mid b \in \bJ \}$ and $d(\bJ):=\sup \{d \mid d \in \bJ\}$, respectively.
The lifetime $d(\bJ) - b(\bJ)$ is called the {\em persistence} of the birth-death pair $x=(b(\bJ), d(\bJ))$ and denoted by $\pers(x):=d(\bJ) - b(\bJ)$.
For an interval decomposable persistence module $\lU \cong \bigoplus_{i \in L} \lI_{\bJ_{i}}$, the {\em persistence diagram} of $\lU$ is defined by a multiset\footnote{A {\em multiset} is a set with multiplicity of each point. 
Note that the collection of birth-death pairs should be a multiset because an interval decomposition of $\lU$ can contain several intervals with the same birth-death pairs.} composed of all birth-death pairs in $\lU$ and denoted by
\[
D(\lU)=\braces{ (b(\bJ_{i}),d(\bJ_{i})) \remiddle i \in L}.
\]
For the persistent homology $H_{q}(\Filt)$ of a filtration $\Filt$, the persistence diagram $D(H_{q}(\Filt))$ is denoted by $D_{q}(\Filt)$ for short. 
For a finite point set $X$ in a metric space $(M, d_{M})$, we define the {\em ball model filtration} $\lB(X)=\{B(X;a):=\bigcup_{\bm{x} \in X}B(\bm{x};a)\}_{a \in \lR}$, where $B(\bm{x};a)=\{\bm{z} \in M \mid d_{M}(\bm{x}, \bm{y}) \leq a\}$ $(a \geq 0)$ and $B(\bm{x};a)=\emptyset$ otherwise.
Figure \ref{fig:ball_pd} shows the persistence diagram $D_{1}(\lB(X))$.

Since each death time is greater than or equal to the corresponding birth time, all birth-death pairs of $D(\lU)$ lie on the region {\em above the diagonal} $\ad:=\{(b,d) \in \lR^{2} \mid b \leq d\}$.\footnote{Precisely speaking, the definition of the birth and death time can contain $-\infty$ and $\infty$.
However, in practical, we can assume that all birth and death times take neither $\infty$ nor $-\infty$ (for more details, please see Section 2.1.2 in \cite{KFH18}).}
For this reason, we define a {\em generalized persistence diagram} by a countable multiset of points in $\ad$, and the set of all generalized persistence diagrams is denoted by $\cD^{g}$.

Let $\DD:=\{(a,a) \mid a \in \lR\}$ denote the diagonal set with all points having infinite multiplicity.
As a distance between persistence diagrams $D$ and $E$, we use the {\em bottleneck distance} which is defined by 
\[
\dW{\infty}(D,E)=\inf_{\gamma} \sup_{x \in D \cup \DD} \norm{x-\gamma(x)}_{\infty},
\]
where $\gamma:D \cup \DD \to E \cup \DD$ is a multi-bijection.\footnote{By considering infinite multiplicity of the diagonal set $\DD$, there always exists a multi-bijection from $D \cup \DD$ to $E \cup \DD$. The bottleneck distance is also called $\infty$-Wasserstein distance.}
When we define the set of persistence diagram as $\cD_{\infty}=\{D \in \cD^{g}\mid \dW{\infty}(D,\DD)<\infty\}$, then $(\cD_{\infty},\dW{\infty})$ becomes a pseudometric space.

For persistence diagrams $D$ and $E$, we define an equivalence $D \sim E$ if $\dW{\infty}(D,E)=0$, and then the difference in the equivalence class is only on the diagonal set, which can be ignored in data analysis.
Thus, we abuse the notation and let $(\cD_{\infty},\dW{\infty})$ also denote the metric space of equivalence classes under $\sim$.

From now on, let $\cD$ be a subspace in $\cD_{\infty}$.
In order to develop statistical analysis for persistence diagrams, it is convenient to focus on not $\cD_{\infty}$ but a subspace $\cD$ in $\cD_{\infty}$.
For example, when we analyze finite point sets in a metric space $(M, d_{M})$ and compute their persistence diagrams via the ball model filtration, all resulting persistence diagrams are in 
\[
\cD(\lB(M)):=\{D_{q}(\lB(X)) \mid X \mbox{ is a finite point set in } M, ~ q \in \lN \cup \{0\}\},
\]
and $\cD(\lB(M))$ is a subspace in $\cD_{\infty}$.
Let $\cF(M)$ denote the set of finite point sets in $M$, then $\cF(M)$ is a metric space with the {\em Hausdorff distance}, which is defined by 
\[
\dH(X,Y):= \max \braces{ \sup_{\bm{x} \in X} \Inf_{\bm{y} \in Y} d_{M}(\bm{x},\bm{y}), \sup_{\bm{y} \in Y} \Inf_{\bm{x} \in X} d_{M}(\bm{x},\bm{y}) } ~ (X,Y \in \cF(M)).
\]
As an important property for a persistence diagram, it is shown that the map $\phi:(\cF(M), \dH) \to (\cD(\lB(M)), \dW{\infty}), X \mapsto D_{q}(\lB(X))$ is Lipschitz continuous, which is known as the {\em stability of persistence diagrams}:
\begin{theorem}[\cite{CdSO14,CEH07}]
\label{thm:stability_ball}
For any two finite point sets $X$ and $Y$ in a metric space $M$, we have
\[
\dW{\infty}(D_{q}(\lB(X)),D_{q}(\lB(Y))) \leq \dH(X,Y).
\]
\end{theorem}

\subsection{Persistence Weighted Kernel}
\label{subsec:pwgk}
In this section, we review the basics of a positive definite kernel and the PWK vector.
We refer the reader to \cite{BTA11,PR16} for the kernel method, which is the statistical theory of a positive definite kernel, and \cite{KFH18} for the PWK vector.

Let $\cX$ be a set.
A function $k:\cX \times \cX \to \lR$ is called a {\em positive definite kernel} on $\cX$ if $k$ is symmetric, i.e., $k(x,y)=k(y,x)$ for any $x,y \in \cX$, and the matrix $(k(x_{i},x_{j}))_{i,j=1,\ldots,\ell}$ is a nonnegative definite matrix for any finite number of $x_{1},\ldots,x_{\ell} \in \cX$.
It is known from the Moore-Aronszajn theorem that a positive definite kernel $k$ uniquely defines a Hilbert space $\cH_{k}$ as a subspace of a real-valued functional space on $\cX$, which is called the {\em reproducing kernel Hilbert space} (RKHS).
\begin{theorem}[Moore-Aronszajn, Theorem 2.14 in \cite{PR16}]
\label{thm:moore}
Let $\cX$ be a set and $k$ be a positive definite kernel on $\cX$.
Then, there uniquely exists a reproducing kernel Hilbert space $\cH_{k}$ satisfying the following:
\begin{description}
\item[(1)] $k(\cdot,x) \in \cH_{k}$ for any $x \in \cX$,
\item[(2)] $\Span \braces{ k(\cdot,x) \remiddle x \in \cX}$ is dense in $\cH_{k}$ with the uniform norm $\norm{\cdot}_{\infty}$,
\item[(3)] $\inn{f}{k(\cdot,x)}_{\cH_{k}}=f(x)$ for any $x \in \cX$ and any $f \in \cH_{k}$.
\end{description}
\end{theorem}
We remark that the inner product of $k(\cdot,x)$ and $k(\cdot,y)$ is given by $\inn{k(\cdot,y)}{k(\cdot,x)}_{\cH_{k}}=k(x,y)$ from the property (3) in Theorem \ref{thm:moore}, and the norm of $k(\cdot,x)$ is given by $\norm{k(\cdot,x)}_{\cH_{k}}=\sqrt{k(x,x)}$.
For later use, we remark that the following inequality holds for any $f,g \in \cH_{k}$:
\begin{align}
\norm{f-g}_{\infty}
&=\sup_{z \in \cX} \abs{f(z) -g(z)} \nonumber \\
&=\sup_{z \in \cX} \abs{\inn{f-g}{k(\cdot,z)}_{\cH_{k}}}\nonumber \\
&\leq\sup_{z \in \cX} \norm{k(\cdot,z)}_{\cH_{k}} \norm{f-g}_{\cH_{k}}\nonumber \\
&=\sup_{z \in \cX} \sqrt{k(z,z)} \norm{f-g}_{\cH_{k}} \label{eq:kernel_ineq}
\end{align}

We call a positive definite kernel $k$ on $\lR^{2}$ {\em bounded} if there exists a constant $\Bdd(k)>0$ such that $\sup_{x,y \in \lR^{2}} \abs{k(x,y)} \leq \Bdd(k)$, and {\em Lipschitz continuous} if there exists a constant $\Lip(k)>0$ such that 
\begin{align*}
\norm{k(\cdot,x)-k(\cdot,y)}_{\cH_{k}} \leq \Lip(k) \norm{x-y}_{\infty} \mbox{ for any } x,y \in \lR^{2}.
\end{align*}
The following proposition is not used here, but will be used in Section \ref{sec:limit_theorems}.
\begin{proposition}
\label{prop:separable}
A reproducing kernel Hilbert space of a bounded Lipschitz continuous positive definite kernel $k$ on $\lR^{2}$ is separable.
\end{proposition}

\begin{proof}
For any $x, y \in \lR^{2}$, we have
\begin{align*}
\norm{k(\cdot,x) -k(\cdot,y)}_{\infty}
&=\sup_{z \in \lR^{2}} \sqrt{k(z,z)} \norm{k(\cdot,x) -k(\cdot,y)}_{\cH_{k}} ~~ (\mbox{Equation \eqref{eq:kernel_ineq}}) \\
&\leq \sqrt{\Bdd(k)} \Lip(k) \norm{x-y}_{\infty} \\
&\leq \sqrt{\Bdd(k)} \Lip(k) \norm{x-y}~~ (\norm{x}_{\infty} \leq \norm{x} \mbox{ for any } x \in \lR^{2}).
\end{align*}
When we use a metric derived from the uniform norm $\norm{\cdot}_{\infty}$ as a metric of $\cH_{k}$, the map $x \in \lR^{2} \mapsto k(\cdot,x) \in \cH_{k}$ is Lipschitz continuous from the above inequality.
Since $\lR^{2}$ is separable, $\Span \{k(\cdot,x) \mid x \in \lR^{2}\}$ is separable from the Lipschitz continuity of $k(\cdot,x)$, and hence $\cH_{k}$ is also separable because $\Span \{k(\cdot,x) \mid x \in \lR^{2}\}$ is dense in $\cH_{k}$ with the uniform norm from Theorem \ref{thm:moore}.
\qed
\end{proof}

Let $\cD$ be a subspace in $\cD_{\infty}$.
We define a {\em weight function} $w:\lR^{2} \to \lR$ for $\cD$ if it satisfies $w(x) = 0$ for any $x \in \DD$ and there exists a constant $\Bdd(w;\cD)>0$ such that $\sup_{D \in \cD}\sum_{x \in D} \abs{w(x)} \leq \Bdd(w;\cD)$.
If $k$ is a measurable bounded positive definite kernel on $\lR^{2}$ and $w$ is a weight function for $\cD$, for a persistence diagram $D \in \cD$, the weighted sum
\[
V^{k,w}(D):= \sum_{x \in D}w(x)k(\cdot, x)
\]
is an element in the RKHS $\cH_{k}$.
We call $V^{k,w}(D)$ the {\em persistence weighted kernel vector} ({\em PWK vector}) \cite{KFH16,KFH18} of $D$ by $k$ and $w$.
We remark that the inner product is given by
\begin{align}
\inn{V^{k,w}(D)}{V^{k,w}(E)}_{\cH_{k}}=\sum_{x \in D}\sum_{y \in E}w(x)w(y)k(x,y) ~~ (D,E \in \cD). \label{eq:inner_pwgk}
\end{align}
Furthermore, for any $D \in \cD$, we have
\begin{align}
\norm{V^{k,w}(D)}_{\cH_{k}}^{2}
\leq \Bdd(k) \pare{ \sum_{x \in D} w(x)}^{2}
\leq \Bdd(k) \Bdd(w;\cD)^{2}, \label{eq:norm_pwk}
\end{align}
and thus the norm of the PWK vector is always finite.

Here, we define classes for positive definite kernels and weight functions as follows:
\[
\Kc:=\{k\mbox{: measurable positive definite kernel on } \lR^{2} \mid {}^{\exists} \Bdd(k)>0, ~ {}^{\exists} \Lip(k)>0\}.
\]
Let $w$ be a weight function for $\cD$.
The class of weight functions such that there exists a constant $L_{\infty}(w;\cD)>0$ satisfying 
\begin{align}
\sum_{x \in D \cup \DD} \abs{w(x)-w(\gamma(x))} \leq L_{\infty}(w;\cD)\sup_{x \in D \cup \DD} \norm{x- \gamma(x)}_{\infty} \label{eq:weight}
\end{align}
for any $D,E \in \cD$ and any multi-bijection $\gamma:D \cup \DD \to E \cup \DD$ is denoted by
\[
\cW_{\infty}(\cD):=\{w: \mbox{weight function for } \cD \mid {}^{\exists} L_{\infty}(w;\cD)>0 ~ \mbox{ satisfying \eqref{eq:weight}}\}
\]Then, the stability theorem for a PWK vector is shown as follows:
\begin{theorem}[\cite{KFH18}]
\label{thm:kernel_stability}
Let $k \in \Kc$ and $w \in \cW_{\infty}(\cD)$.
Then, for any persistence diagrams $D, E \in \cD$, we have
\[
\norm{V^{k,w}(D)-V^{k,w}(E)}_{\cH_{k}} \leq (\Lip(k)\Bdd(w;\cD)+\sqrt{\Bdd(k)}L_{\infty}(w;\cD)) \dW{\infty}(D,E).
\]
\end{theorem}
It is also shown in \cite{KFH18} that the Gaussian kernel $k(x,y)=e^{-\frac{\norm{x-y}^{2}}{2\sigma^{2}}}$ is in $\Kc$ and the arctangent type weight function $w_{{\rm arc}}(x)=\arctan(C_{{\rm arc}} \pers(x)^{p_{{\rm arc}}})$ $(C_{{\rm arc}}>0, p_{{\rm arc}} \in \lN)$ is in $\cW_{\infty}(\cD(\lB(M)))$ where $M$ is a triangulable compact subspace in $\lR^{d}$ and $p_{{\rm arc}}>d+1$. 
In other words, for finite point sets $X$ and $Y$ in a triangulable compact subspace $M \subset \lR^{d}$, if $p_{{\rm arc}}>d+1$, then the map $X \mapsto V^{\kG, w_{{\rm arc}}}(D_{q}(\lB(X)))$ is shown to be Lipschitz continuous from Theorem \ref{thm:kernel_stability}.

\section{Limit theorems}
\label{sec:limit_theorems}
In this section, we will regard a persistence diagram as a random value taking values in a metric space $\cD$ and consider the expectation of the PWK vector in the RKHS.
Before discussing probabilistic properties of a persistence diagram, we briefly review the basic of probability theory in a metric space, following \cite{Du10,Vaart98}, and probability theory in a Banach space, following \cite{LT13}.
Throughout this section, we fix $(\Omega, \cA, \prob)$ a probability space.

\subsection{Probability in a metric space}
When $(\cX,\cB_{\cX})$ is a measurable space, a measurable map $X:(\Omega, \cA) \to (\cX, \cB_{\cX})$ is called an $\cX$-valued random element.
If $\cX$ is a topological space and $\cB_{\cX}$ is defined by the Borel $\sigma$-set of $\cX$, an $\cX$-valued random element $X$ is called a Borel random element.
An $\cX$-valued random element $X$ induces a probability measure $P_{X}$ from $\prob$ as $P_{X}(A):=\prob(X^{-1}(A)) ~ (A \in \cB_{\cX})$, and the probability measure $P_{X}$ is called the probability distribution of $X$.
Then, we say that $X$ is drawn from $P_{X}$ and write $X \sim P_{X}$.
When $\cX$-valued random elements $X_{1},\ldots,X_{n}$ are drawn from the same probability distribution $P$ and are independent, we say that $X_{1},\ldots,X_{n}$ are independently identically distributed from $P$ and this is denoted by $X_{1},\ldots,X_{n}, \iid \sim P$.
Let $\varphi:\cX \to \lR$ be a measurable map, then the integral $\int_{\Omega}\varphi(X(\omega))d\prob(\omega)$ is called the expectation of $\varphi(X)$ and denoted by $\lE[\varphi(X)]$.
When we emphasize the probability distribution $P$ of $X$, the expectation $\lE[\varphi(X)]$ is also denoted by $\lE_{X \sim P}[\varphi(X)]$.
An $\lR$-valued random element is simply called a random variable.
We consider an expectation for a map in Section \ref{sec:confidence}, which may not be a random variable, and here define the outer expectation of a map $U:\Omega \to \lR$ by
\[
\lE^{*}[U]:=\inf \{\lE[V] \mid V \mbox{ is a random variable, } V \geq U, ~ \lE[V] \mbox{ exists}\}.
\]

Here, we define stochastic convergences for a sequence of maps in a metric space $(\cX, d)$.
Let $\{X_{n}:\Omega \to \cX\}_{n \in \lN}$ be a sequence of maps and $X$ be an $\cX$-valued random element.
Then, definitions of standard convergences for random variables are generalized to ones for maps taking values in a metric space as follows:
\begin{itemize}
\item If there exists a sequence of random variables $\{\DD_{n}\}_{n \in \lN}$ satisfying $d(X_{n},X)\leq \DD_{n}$ for any $n \in \lN$ and $\prob(\lim_{n \to \infty} \DD_{n} = 0)=1$, then $\{X_{n}\}_{n \in \lN}$ is said to converge to $X$ {\em almost surely} and it is denoted by $X_{n} \convas X$.
\item If $\lim_{n \to \infty}\lE^{*}[\varphi(X_{n})]=\lE[\varphi(X)]$ for all $1$-Lipschitz function $\varphi:\cX \to \lR$, then $\{X_{n}\}_{n \in \lN}$ is said to converge {\em in distribution} to $X$ and it is denoted by $X_{n} \convd X$.
\end{itemize}

\subsection{Probability for a Banach space}
Let $(B, \norm{\cdot}_{B})$ be a Banach space and $V:\Omega \to B$ be a $B$-valued Borel random element.
If there exists an element $W \in B$ satisfying $\int_{\Omega} f(V(\omega))d\prob(\omega) =f(W)$ for any $f \in B^{*}$, where $B^{*}$ be the topological dual space\footnote{$B^{*}$ is the set of all continuous linear real-valued functions $f:B \to \lR$.}  of $B$, the element $W$ is called the Pettis integral of $V$ and denoted by $\lE[V]$.
If $V$ is Radon, the expectation $\lE[V]$ always exists and it satisfies $\norm{\lE[V]}_{B} \leq \lE[\| V \|_{B}]$.\footnote{We call $V:\Omega \to B$ {\em Radon} if, for any $\ee>0$, there exists a compact set $K$ in the Borel $\sigma$-set of $B$ such that $\prob(V \in K) \geq 1-\ee$. For more details, please see Section 2.1 in \cite{LT13}.}
A $B$-valued Radon random element $G$ is called a {\em centered Gaussian random element} if $f(G)$ is a real valued Gaussian random variable with mean zero for any $f \in B^{*}$.
A centered Gaussian random element $G$ is determined by its covariance structure $C_{G}:B^{*} \times B^{*} \to \lR$ which is defined by $C_{G}(f,g):=\lE[f(G)g(G)]$ $(f,g \in B^{*})$. 

For $B$-valued random elements $V_{1},\ldots,V_{n}$, the sum is denoted by $S_{n}:=\sum_{i=1}^{n}V_{i}$.
The strong law of large numbers and the central limit theorem are stated as follows:
\begin{theorem}[Strong law of large numbers, Theorem 10.5 in \cite{LT13}]
\label{thm:SLNN_banach}
Let $B$ be a separable Banach space, $V$ be a $B$-valued Radon random element, and $\{V_{n}\}_{n \in \lN}$ be a sequence of independent $B$-valued random elements distributed as $V$.
Then, $\lE[V] =0$ and $\lE[\| V \|_{B}] < \infty$ if and only if $n^{-1}S_{n} \convas 0$.
\end{theorem}

\begin{theorem}[Central limit theorem, Corollary 10.9 in \cite{LT13}]
\label{thm:CLT_banach}
Let $B$ be a separable Banach space of type 2,\footnote{We do not define the concept of type 2 in this paper because a Hilbert space is of type 2 and a Banach space which will be used in this paper is a reproducing kernel Hilbert space. For more details, please see Section 9.2 in \cite{LT13}.} $V$ be a $B$-valued Radon random element, and $\{V_{n}\}_{n \in \lN}$ be a sequence of independent $B$-valued random elements distributed as $V$.
If $\lE[V] =0$ and $\lE[\| V \|_{B}^{2}] < \infty$, $n^{-1/2}S_{n} \convd G$ where $G$ is a centered Gaussian random variable with the same covariance structure as $V$.
\end{theorem}

\subsection{Limit theorems for the PWK vectors}
Let $\cD$ be a subspace in $\cD_{\infty}$ and $D$ be a $\cD$-valued Radon random element.
If $V^{k,w}:\cD \to \cH_{k}$ is continuous, the PWK vector $V^{k,w}(D)$ of a $\cD$-valued Radon random element $D$ can be viewed as an $\cH_{k}$-valued Radon random element, and the expectation $\lE[V^{k,w}(D)]$ of the PWK vector is well defined in the RKHS $\cH_{k}$.
From now on, we assume that a positive definite kernel $k$ and a weight function $w$ are selected to make $V^{k,w}:\cD \to \cH_{k}$ continuous.
In fact, Theorem \ref{thm:kernel_stability} ensures that there exist $k$ and $w$ such that $V^{k,w}:\cD \to \cH_{k}$ is continuous.

Let $k \in \Kc$ be a Lipschitz continuous positive definite kernel on $\lR^{2}$.
Then, the RKHS $\cH_{k}$ is separable from Proposition \ref{prop:separable}, which satisfies the assumptions in the strong law of large numbers and the central limit theorem for a Banach space (Theorem \ref{thm:SLNN_banach} and Theorem \ref{thm:CLT_banach}).
For any measurable bounded positive definite kernel $k$ and a weight function $w$ for $\cD$, which are needed to define the PWK vector, we have confirmed in Equation \eqref{eq:norm_pwk} that the norm $\| V^{k,w}(D) \|_{\cH_{k}}$ is always bounded from above independent of $D$.
For this reason, it is always satisfied that $\lE[\| V^{k,w}(D) \|_{\cH_{k}}] < \infty$ and $\lE[\| V^{k,w}(D) \|_{\cH_{k}}^{2}] < \infty$.
For $\cD$-valued random element $D_{1},\ldots,D_{n}$, $S^{k,w}_{n}$ denotes the sum $\sum_{i=1}^{n}V^{k,w}(D_{i})$.
Applying Theorem \ref{thm:SLNN_banach} and Theorem \ref{thm:CLT_banach} to $V^{k,w}(D) - \lE[V^{k,w}(D)]$, we obtain the following:
\begin{theorem}[Strong law of large numbers for the PWK vector]
\label{thm:SLNN_pwk}
Let $k \in \Kc$, $w$ be a weight function for $\cD$, $D$ be a $\cD$-valued Radon random element, and $\{D_{n}\}_{n \in \lN}$ be a sequence of independent $\cD$-valued random elements distributed as $D$.
Then, we have
\[
\frac{1}{n} \sum_{i=1}^{n} V^{k,w}(D_{i})  \convas \lE[V^{k,w}(D)].
\]
\end{theorem}

\begin{theorem}[Central limit theorem for the PWK vector]
\label{thm:CLT_pwk}
Let $k \in \Kc$, $w$ be a weight function for $\cD$, $D$ be a $\cD$-valued Radon random element, and $\{D_{n}\}_{n \in \lN}$ be a sequence of independent $\cD$-valued random elements distributed as $D$.
Then, we have 
\[
\frac{1}{\sqrt{n}} \sum_{i=1}^{n} \pare{ V^{k,w}(D_{i}) - \lE[V^{k,w}(D)]} \convd G^{k,w}(D),
\]
where $G^{k,w}(D)$ is the centered Gaussian random variable with the same covariance structure as $V^{k,w}(D)$.
\end{theorem}

\section{Confidence interval}
\label{sec:confidence}
Since the RKHS $\cH_{k}$ is a subspace of the real-valued functional space on $\lR^{2}$, the PWK vector $V^{k,w}(D) \in \cH_{k}$ is a real-valued function on $\lR^{2}$ by
\[
z \in \lR^{2} \mapsto V^{k,w}(D)(z):=\sum_{x \in D}w(x)k(z,x) \in \lR.
\]
Let $P$ be a Radon probability measure on $\cD$ and $D \sim P$, then the expectation $\lE[V^{k,w}(D)]$ is a real-valued function on $\lR^{2}$, and its value at $z \in \lR^{2}$ is denoted by $\lE[V^{k,w}(D)](z) \in \lR$.
From the strong law of large numbers for the PWK vector (Thereom \ref{thm:SLNN_pwk}), we have confirmed that 
\[
\norm{ \frac{1}{n} \sum_{i=1}^{n}  V^{k,w}(D_{i}) -  \lE[V^{k,w}(D)]}_{\cH_{k}} \convas 0
\]
where $D_{1},\ldots, D_{n}, \iid \sim P$.
Furthermore, the uniform convergence
\[
\sup_{z \in \lR^{2}} \abs{ \frac{1}{n} \sum_{i=1}^{n} V^{k,w}(D_{i})(z) -  \lE[V^{k,w}(D)](z)} \convas 0 
\]
follows from Equation \eqref{eq:kernel_ineq}.
Since the i.i.d. number $n$ is finite in practice, to measure how the sample mean $n^{-1}\sum_{i=1}^{n} V^{k,w}(D_{i})(z)$ is close to the expectation $\lE[V^{k,w}(D)](z)$, we will estimate a number $\xi_{n,\alpha} \in \lR$ satisfying the following inequality:
\begin{align}
\prob\pare{ \abs{ \frac{1}{n} \sum_{i=1}^{n} V^{k,w}(D_{i})(z) -  \lE[V^{k,w}(D)](z) } \leq \xi_{n,\alpha} \mbox{ for any } z \in \lR^{2}} \geq 1-\alpha. \label{eq:confi_xi}
\end{align}
This estimation is based on the concept of a uniform confidence band for a stochastic process when we see $\{\lE[V^{k,w}(\cdot)](z):\cD \to \lR\}_{z \in \lR^{2}}$ as a stochastic process.
We first review a uniform confidence band for a stochastic process, following \cite{Ko08,Vaart98}, and then construct the uniform confidence band for the expectation of a PWK vector.

\subsection{Review of a uniform confidence band}
\label{subsec:confi_basic}
Let $(\cX, \cB_{\cX})$ be a measurable space, $X:\Omega \to \cX$ be an $\cX$-valued random element, $P$ be the distribution of $X$, $\theta$ be a real valued parameter attached to $P$, and $X_{1},\ldots,X_{n}, \iid \sim P$.
For simplicity, the collection of the random elements $X_{1},\ldots,X_{n}$ is denoted by $\bX_{n}:=(X_{1},\ldots,X_{n})$.
In the concept of interval estimation, we construct an interval made from $\bX_{n}$ which contains $\theta$ with high probability.
To be precise, for $\alpha \in (0, 1)$ and two statistics $L(\bX_{n})$ and $U(\bX_{n})$ satisfying $L(\bX_{n}) \leq U(\bX_{n})$, an interval $[L(\bX_{n}), U(\bX_{n})]$ is called a {\em confidence interval} for $\theta$ at level $1-\alpha$ if it satisfies $\prob(\theta \in [L(\bX_{n}), U(\bX_{n})]) \geq 1-\alpha.$

\begin{example}
\label{ex:confidence_normal}
Let $X_{1},\ldots,X_{n}, \iid \sim \cN(\mu,\sigma^{2}_{0})$,\footnote{$\cN(\mu,\sigma^{2})$ denotes the normal distribution with mean $\mu \in \lR$ and variance $\sigma^{2}>0$.} where $\mu$ is unknown and $\sigma^{2}_{0}$ is known, and $\hat{\eta}(\bX_{n}):=(\sum_{i=1}^{n}X_{i} -\mu)/ \sqrt{n}\sigma_{0}$ be a statistics.
Since $\hat{\eta}(\bX_{n}) \sim \cN(0, 1)$, we have $\prob\pare{\hat{\eta}(\bX_{n}) \leq c}=(2\pi)^{-1/2}\int_{-\infty}^{c} \exp(-z^{2}/2)dz$.
Then, for the upper $\alpha$-quantile $c_{\alpha}$ satisfying $\Phi(c_{\alpha})=1-\alpha$, $[\overline{X} - \pare{\sigma_{0}/\sqrt{n}}c_{\alpha/2}, \overline{X} + \pare{\sigma_{0}/\sqrt{n}}c_{\alpha/2}]$ is the confidence interval for $\mu$ at $1-\alpha$ because $\prob\pare{ \abs{\hat{\eta}(\bX_{n})} \leq c_{\alpha/2}}=1-\alpha$.
\end{example}

Fortunately, the quantile $c_{\alpha}$ of the distribution $\prob\pare{\hat{\eta}(\bX_{n}) \leq c_{\alpha}}=1-\alpha$ in Example \ref{ex:confidence_normal} is numerically computable from the property of the standard normal distribution.
In the case of a general statistic $\hat{\eta}(\bX_{n})$, its distribution of $\hat{\eta}(\bX_{n})$ may be unknown or the quantile may be hard to compute.
For such a case, the {\em bootstrap method} is a powerful tool to estimate the quantile.

Let $\hat{\eta}(\bm{X}_{n})$ be a statistic of $X_{1},\ldots,X_{n}, \iid \sim P$.
For any $\omega_{0} \in \Omega$, we define a probability distribution $\lP_{n}(\omega_{0})$ on $\cX$ by $\lP_{n}(\omega_{0}):=n^{-1}\sum_{i=1}^{n}\dd_{X_{i}(\omega_{0})}$ where $\dd_{x}$ is the Dirac delta measure at $x \in \cX$.
The random probability measure $\lP_{n}=n^{-1}\sum_{i=1}^{n}\dd_{X_{i}}$ is called the {\em empirical distribution} of $X_{1},\ldots,X_{n}$.
Let $X^{*}$ be a map which transforms $\omega \in \Omega$ to an $\cX$-valued random element.
If $X^{*}(\omega) \sim \lP_{n}(\omega)$ for any $\omega \in \Omega$, $X^{*}$ is called a {\em bootstrap sample} from $\lP_{n}$ and denoted by $X^{*} \sim \lP_{n}$.
Even after we fix $\omega_{0} \in \Omega$, $X^{*}(\omega_{0})$ is still an $\cX$-valued random element, that is, $X^{*}(\omega_{0})(\zeta) \in \lR$ for $\zeta \in \Omega$.
Here, we define a function $\hat{\eta}(\bX^{*}_{n})(\omega_{0}):\Omega \to \lR$ by $(\hat{\eta}(\bX^{*}_{n})(\omega_{0}))(\zeta):=\hat{\eta}(X^{*}_{1}(\omega_{0})(\zeta),\ldots,X^{*}_{n}(\omega_{0})(\zeta))$ $(\zeta \in \Omega)$.

In order to define the convergence in distribution for $\hat{\eta}(\bX^{*}_{n})$, we define the {\em conditional convergence} in distribution.
Let $\{Y_{n} \sim \lP_{n}\}_{n \in \lN}$ be a sequence of bootstrap samples and $T$ be an $\cX$-valued random element.
Recall that the convergence in distribution $Y_{n}(\omega_{0}) \convd T$ for some $\omega_{0} \in \Omega$ is defined by $\lim_{n \to \infty}\lE^{*}[\varphi(Y_{n}(\omega_{0}))]=\lE[\varphi(T)]$ for all $1$-Lipschitz function $\varphi:\cX \to \lR$.
Thus, in a similar way, we define that a sequence $\{Y_{n} \sim \lP_{n}\}_{n \in \lN}$ converges in distribution to $T$ {\em conditionally on} $X_{1},\ldots,X_{n}$ if a function $H_{n}:\Omega \to \lR$ which is defined by
\[
H_{n}(\omega):=\sup_{\varphi \in {\rm BL}_{1}(\cX)} \abs{\lE^{*}[\varphi(Y_{n}(\omega))]-\lE[\varphi(T)]},
\]
where ${\rm BL}_{1}(\cX)$ is the set of all $1$-Lipschitz functions $\varphi:\cX \to \lR$, satisfies $H_{n} \convas 0$.
Then, the conditional convergence is denoted by $Y_{n} \mid \bX_{n} \convd T$.

For $\omega \in \Omega$, $\hat{\xi}^{*}_{n, \alpha}(\omega)$ denotes the upper $\alpha$-quantile of $\hat{\eta}(\bX^{*}_{n})(\omega)$, and $\hat{\xi}^{*}_{n, \alpha}$ is treated as a real-valued function on $\Omega$.
The following theorem is obtained by a slight modification of the proof of Lemma 23.3 in \cite{Vaart98}:
\begin{theorem}[\cite{Vaart98}]
\label{thm:asymptotic_efron}
Let $X_{1},\ldots, X_{n}, \iid \sim P$ and $T$ be a random value with continuous distribution function $F$.
If $\hat{\eta}(\bX_{n}) \convd T$ and $\hat{\eta}(\bX^{*}_{n}) \mid \bX_{n} \convd T$, then for almost all $\alpha \in (0,1)$, we have
\begin{align}
\label{eq:asymp_efron}
\lim_{n \to \infty} \prob\pare{\braces{\omega \in \Omega \remiddle \hat{\eta}(\bX_{n})(\omega) \leq \hat{\xi}^{*}_{n, \alpha}(\omega) }}= 1-\alpha,
\end{align}
where $\hat{\xi}^{*}_{n, \alpha}(\omega)$ is the upper $\alpha$-quantile of $\hat{\eta}(\bX^{*}_{n})(\omega)$ for $\omega \in \Omega$.
\end{theorem}

Since the upper $\alpha$-quantile $\hat{\xi}_{n, \alpha}$ of $\hat{\eta}(\bX_{n})$ is defined by $\prob(\hat{\eta}(\bX_{n}) \leq \hat{\xi}_{n, \alpha}) = 1-\alpha$, Theorem \ref{thm:asymptotic_efron} ensures that $\hat{\xi}^{*}_{n, \alpha}$ in Equation \eqref{eq:asymp_efron} approximates $\hat{\xi}_{n, \alpha}$ in an asymptotic sense.

As mentioned above, we will view a family $\{\lE[V^{k,w}(\cdot)](z):\cD \to \lR\}_{z \in \lR^{2}}$ of values of the expectation of a PWK vector as a stochastic process.
Here, we explain the bootstrap method for a general stochastic process.

Let $T$ be a set. 
A family $\lX=\{\lX(t) :\Omega \to \lR\}_{ t \in T}$ of random variables is called a {\em stochastic process} indexed by $T$.
A stochastic process $\lG=\{\lG(t)\}_{t \in T}$ indexed by $T$ is called the {\em Gaussian process} if, for any finite subset $T_{d}:=\{t_{1},\ldots,t_{d}\} \subset T$, the distribution of the multivariate vector $(\lG(t_{1}), \ldots, \lG(t_{d}))^{T}$ is the $d$-dimensional normal distribution $\cN_{d}(0,\bm{\Sigma}_{T_{d}})$ where the covariance matrix is given by $(\bm{\Sigma}_{T_{d}})_{i,j}=\lE[\lG(t_{i})\lG(t_{j})]$ for $i,j =1,\ldots,d$.
Let $\cF:=\{f:\cX \to \lR\}$ be a family of real-valued measurable functions on $\cX$ and $Pf:=\int_{\cX} f(x)dP(x)$ denote the expectation of $f \in \cF$.
Then, the family $\{Pf\}_{f \in \cF}$ of the expectations is a stochastic process over $\cF$.
Let $X_{1},\ldots, X_{n}, \iid \sim P$, $\lP_{n}$ be the empirical distribution of $X_{1},\ldots,X_{n}$, and $\lG_{n}f$ denote $\sqrt{n}(\lP_{n}f-Pf)$.
Then, the stochastic process $\lG_{n}=\{\lG_{n}f \}_{f \in \cF}$ indexed by $\cF$ is called the {\em empirical process} of $X_{1},\ldots,X_{n}$.
In order to define a convergence of $\lG_{n}$, we define a space $\ell^{\infty}(\cF)$ of bounded\footnote{We call $\lX$ bounded if $\prob(\lX f<\infty)=1$ for any $f \in \cF$.} stochastic processes over $\cF$ and use a metric derived from the uniform norm $\norm{\lX}_{\infty}:=\sup_{f \in \cF} \abs{\lX f}$ as a metric on $\ell^{\infty}(\cF)$.
If $\lG_{n} \convd \lG$ in $\ell^{\infty}(\cF)$ where the covariance structure of $\lG$ is given by $\lE[\lG f \lG g]=Pfg - Pf Pg$ $(f,g \in \cF)$, $\cF$ is said to be {\em $P$-Donsker}.

For bootstrap samples $X^{*}_{1},\ldots,X^{*}_{n}, \iid \sim \lP_{n}$, we define the {\em bootstrap empirical distribution} by $\lP^{*}_{n}=n^{-1}\sum_{i=1}^{n}\dd_{X^{*}_{i}}$.
Recall again that, after we fix $\omega_{0} \in \Omega$, $X^{*}_{i}(\omega_{0}) \sim \lP_{n}(\omega_{0})$ is an $\cX$-valued random element and the bootstrap empirical distribution $\lP^{*}_{n}(\omega_{0})=n^{-1}\sum_{i=1}^{n}\dd_{X^{*}_{i}(\omega_{0})}$ is still a random probability distribution.
If $\lG^{*}_{n}(\omega):=\{ \sqrt{n}(\lP^{*}_{n}(\omega)f-\lP_{n}(\omega)f) \}_{f \in \cF}$ is a stochastic process for any $\omega \in \Omega$, we call $\lG^{*}_{n}$ the {\em bootstrap empirical process} of $X^{*}_{1},\ldots,X^{*}_{n}$.

From now on, we will characterize a class of $\cF$ which satisfies $\lG^{*}_{n} \mid \bX_{n} \convd \lG$ in $\ell^{\infty}(\cF)$.
A function $E:\cX \to \lR$ is called the {\em envelope function} of $\cF$ if $\abs{f(x)} \leq E(x) < \infty$ for any $f \in \cF$ and $x \in \cX$.
The envelope function $E$ of $\cF$ is a real-valued function on the probability space $(\cX,\cB_{\cX},P)$ and we define the expectation of $E$ by the outer expectation $\lE^{*}_{X \sim P}[E(X)]$.
Then, the conditional convergence $\lG^{*}_{n} \mid \bX_{n} \convd \lG$ is characterized by the $P$-Donsker condition and the boundedness of the outer expectation of a squared envelope function. 
\begin{proposition}[Theorem 23.7 in \cite{Vaart98}, Theorem 2.7 in \cite{Ko08}]
\label{prop:donsker}
Let $\cF$ be a family of measurable functions on $\cX$ with envelope function $E$.
If $\cF$ is $P$-Donsker and $\lE^{*}_{X \sim P}[E(X)^{2}] < \infty$, then $\lG^{*}_{n} \mid \bX_{n} \convd \lG$ in $\ell^{\infty}(\cF)$.
\end{proposition}
 
Let $\varphi:\ell^{\infty}(\cF) \to \lR$ be a continuous function and $\cF$ satisfy the assumption in Proposition \ref{prop:donsker}, then we have $\varphi(\lG_{n}) \convd \varphi(\lG)$ and $\varphi(\lG_{n}^{*}) \mid \bX_{n} \convd \varphi(\lG)$ from the continuous mapping theorem (Theorem 18.11 in \cite{Vaart98}).
As a continuous function on $\ell^{\infty}(\cF)$, we will use $\psi:\ell^{\infty}(\cF) \to \lR$ which is defined by $\lH \mapsto \sup_{f \in \cF}\abs{\lH f}$.
From Theorem \ref{thm:asymptotic_efron}, we have $\lim_{n \to \infty}\prob( \psi(\lG_{n}) \leq \hat{\xi}^{*}_{n, \alpha})=1-\alpha$ where $\hat{\xi}^{*}_{n, \alpha}(\omega)$ is the upper $\alpha$-quantile of $\psi(\lG_{n}^{*})(\omega)=\sup_{f \in \cF} \abs{ \lG_{n}^{*}(\omega)f}$.
Then, the probability $\prob(\psi(\lG_{n}) \leq \hat{\xi}^{*}_{n, \alpha})$ is given by
\begin{align*}
\prob\pare{ \psi(\lG_{n}) \leq \hat{\xi}^{*}_{n, \alpha}}
&=\prob\pare{ -\hat{\xi}^{*}_{n, \alpha} \leq \sqrt{n}\pare{\lP_{n}f - Pf} \leq \hat{\xi}^{*}_{n, \alpha} \mbox{  for any } f \in \cF} \\
&=\prob\pare{ Pf \in \bracket{\lP_{n}f-\frac{\hat{\xi}^{*}_{n, \alpha}}{\sqrt{n}}, \lP_{n}f+\frac{\hat{\xi}^{*}_{n, \alpha}}{\sqrt{n}}} \mbox{  for any } f \in \cF}.
\end{align*}
Since a realization of $\hat{\xi}^{*}_{n, \alpha} / \sqrt{n}$ is independent of any $f \in \cF$, we call that $\lP_{n}f \pm \hat{\xi}^{*}_{n, \alpha} / \sqrt{n}$ is a {\em uniform confidence band} for $\{ Pf\}_{f \in \cF}$ at level $1-\alpha$.
To sum up, the uniform confidence band for $\{Pf\}_{f \in \cF}$ is obtained as follows:
\begin{theorem}
\label{thm:confi_interval}
Let $X_{1},\ldots, X_{n}, \iid \sim P$ and $\cF$ be a family of measurable functions on $\cX$.
If $\cF$ is $P$-Donsker and the envelope function $E$ of $\cF$ satisfies $\lE^{*}_{X \sim P}[E(X)^{2}] < \infty$, then for almost all $\alpha \in (0,1)$, we have
\[
\lim_{n\to\infty}\prob\pare{ Pf \in \bracket{\lP_{n}f-\frac{\hat{\xi}^{*}_{n, \alpha}}{\sqrt{n}}, \lP_{n}f+\frac{\hat{\xi}^{*}_{n, \alpha}}{\sqrt{n}}} \mbox{ {\rm  for any} } f \in \cF} = 1-\alpha,
\]
where $\hat{\xi}^{*}_{n, \alpha}(\omega)$ is the upper $\alpha$-quantile of $\sup_{f \in \cF} \abs{ \lG_{n}^{*}(\omega)f}$ for $\omega \in \Omega$.
\end{theorem}

In the rest of this section, we introduce a sufficient condition for $\cF$ to be $P$-Donsker.
Let $\ee>0$, $\cG$ be a family of real-valued measurable functions on $\cX$, and $Q$ be a probability measure on $\cX$.
For any $f \in \cF$, if there exists $g \in \cG$ such that $\norm{f-g}_{L_{2}(Q)} \leq \ee$, where the $L_{2}(Q)$-norm is defined by $\norm{h}_{L_{2}(Q)}:=(\int_{\cX}\abs{h(x)}^{2}dQ(x))^{1/2}$, then $\cG$ is said to {\em cover} $\cF$ by $L_{2}(Q)$-balls with radius $\ee$.
A {\em covering number} $N(\ee, \cF, L_{2}(Q))$ of $\cF$ by $L_{2}(Q)$-balls with radius $\ee$ is defined by the minimum number to cover $\cF$:
\[
N(\ee, \cF, L_{2}(Q)):=\min\{\card{\cG} \mid \mbox{$\cG$ covers $\cF$ by $L_{2}(Q)$-balls with radius $\ee$}\}.
\]
If there exists a countable subset $\cG \subset \cF$ such that for any $f \in \cF$ there exists a sequence $\{g_{n}\}_{n \in \lN} \subset \cG$ such that $\lim_{n \to \infty}g_{n}(x) = f(x)$ for all $x \in \cX$, $\cF$ is said to be {\em pointwise measurable}.
Then, a sufficient condition for $P$-Donsker is given as follows:
\begin{proposition}[Theorem 19.14 in \cite{Vaart98}, Theorem 8.19 in \cite{Ko08}]
\label{prop:donsker_condition}
Let $\cF$ be a pointwise measurable family of measurable functions on $\cX$ satisfying
\[
\int_{0}^{1} \sqrt{\log \sup_{Q} N(\ee\norm{F}_{L_{2}(Q)}, \cF, L_{2}(Q))}d\ee<\infty,
\]
where $E$ is an envelope function of $\cF$ and $Q$ is taken over all probability measures such that $\norm{E}_{L_{2}(Q)}>0$.
If $\lE^{*}_{X \sim P}[E(X)^{2}] < \infty$, then $\cF$ is $P$-Donsker.
\end{proposition}

\subsection{Uniform confidence band for the expectation of a PWK vector}
By replacing variables in $V^{k,w}(D)(z)$, we define a new function on $\cD$ by
\[
f^{k,w}_{z}:\cD \to \lR, ~~ D \mapsto V^{k,w}(D)(z)=\sum_{x \in D}w(x)k(z,x).
\]
Then, we have $P f^{k,w}_{z}=\lE_{D \sim P}[V^{k,w}(D)](z)$ and $\lP_{n}f^{k,w}_{z}=n^{-1}\sum_{i=1}^{n} V^{k,w}(D_{i})(z)$.
Here, we restrict the domain of a point $z \in \lR^{2}$ to a region $A \subset \lR^{2}$ and the family of PWK functions on $A$ is denoted by $\cF^{k,w}_{A}:=\{f^{k,w}_{z}: \cD \to \lR \mid z \in A \}$.
As a corollary of Theorem \ref{thm:confi_interval}, we obtain the following:
\begin{corollary}
\label{cor:confi_interval_pwk_cor}
Let $k$ be a measurable bounded positive definite kernel on $\lR^{2}$, $w$ be a weight function for $\cD$, $P$ be a Radon probability measure on $\cD$, $D_{1},\ldots, D_{n}, \iid \sim P$, and $A$ be a subset in $\lR^{2}$.
If $\cF^{k,w}_{A}$ is $P$-Donsker and the envelope function $E$ of $\cF^{k,w}_{A}$ satisfies $\lE^{*}_{D \sim P}[E(D)^{2}] < \infty$, for almost all $\alpha \in (0,1)$, we have
\[
\lim_{n\to\infty}\prob\pare{ P f^{k,w}_{z} \in \bracket{ \lP_{n}f^{k,w}_{z}-\frac{\hat{\xi}^{*}_{n, \alpha}}{\sqrt{n}}, \lP_{n}f^{k,w}_{z}+\frac{\hat{\xi}^{*}_{n, \alpha}}{\sqrt{n}}} \mbox{ {\rm  for any} } z \in A} \geq 1-\alpha,
\]
where $\hat{\xi}^{*}_{n, \alpha}(\omega)$ is the upper $\alpha$-quantile of $\sup_{z \in A} \abs{ \lG_{n}^{*}f^{k,w}_{z}(\omega)}$ for $\omega \in \Omega$.
\end{corollary}

The assumption for $\cF^{k,w}_{A}$ in Corollary \ref{cor:confi_interval_pwk_cor} is satisfied by a Lipschitz continuous positive definite kernel and a bounded dense subset in $\lR^{2}$.
\begin{proposition}
\label{prop:pwgk_donsker}
Let $k \in \Kc$, $w$ be a weight function for $\cD$, $P$ be a Radon probability measure on $\cD$, and $A$ is a bounded dense subset in $\lR^{2}$.
Then, $\cF^{k,w}_{A}$ is $P$-Donsker and there exists an envelope function $E$ of $\cF^{k,w}_{A}$ satisfying $\lE^{*}_{D \sim P}[E(D)^{2}] < \infty$.
\end{proposition}

\begin{proof}
We define the envelope function of $\cF^{k,w}_{A}$ as $E(D) \equiv \Bdd(w;\cD) \Bdd(k)$ because it satisfies 
\[
\abs{f^{k,w}_{z}(D)}= \abs{\sum_{x \in D}w(x) k(z,x)} \leq \Bdd(w;\cD) \Bdd(k) =E(D) < \infty
\]
for any $z \in A$ and $D \in \cD$.
Then, the outer expectation of $E^{2}$ is bounded from above because $\lE^{*}_{D \sim P}[E(D)^{2}] \equiv (\Bdd(w;\cD) \Bdd(k))^{2} < \infty$.

For any $y,z \in \lR^{2}$, $D \in \cD$, and probability measure $Q$ on $\cD$, since 
\begin{align}
\abs{f^{k,w}_{y}(D)-f^{k,w}_{z}(D)}
&=\abs{\sum_{x \in D}w(x)k(y,x) - \sum_{x \in D}w(x)k(z,x)} \nonumber \\
&\leq \sum_{x \in D}\abs{w(x)} \abs{k(y,x) - k(z,x)}  \nonumber \\
& \leq \Bdd(w;\cD) \norm{k(y,\cdot) - k(z,\cdot)}_{\infty} \nonumber \\
&\leq \Bdd(w;\cD) \sqrt{\Bdd(k)} \norm{k(y,\cdot) - k(z,\cdot)}_{\cH_{k}} ~~ (\mbox{Equation \eqref{eq:kernel_ineq}})  \nonumber \\
&\leq \Bdd(w;\cD) \sqrt{\Bdd(k)} \Lip(k)\norm{y-z}_{\infty}, \label{eq:conti_donsker}
\end{align}
we have
\begin{align}
\norm{f^{k,w}_{y}-f^{k,w}_{z}}_{L_{2}(Q)} 
&\leq \Bdd(w;\cD) \Lip(k)\sqrt{\Bdd(k)}\norm{y-z}_{\infty}. \label{eq:donsker_1}
\end{align}
With respect to the bounded region $A$, there exist $a=(a_{1},a_{2}) \in \lR^{2}$ and $T>0$ such that $\tilde{A}=[a_{1}, a_{1} + T] \times [a_{2}, a_{2} + T]$ contains $A$.
For arbitrary small $\dd>0$, we assume that $N:=T/\dd$ is an integer without loss of generality.
Let $p_{i,j}:=(a_{1} + (i-1)\dd, a_{2} + (j-1)\dd) ~ (i,j=1,\ldots,N)$ be a grid point of $\tilde{A}$.
For any $z \in A$ and any probability measure $Q$ on $\cD$, there exists a grid point $p_{i,j}$ of $\tilde{A}$ such that $\norm{z-p_{i,j}}_{\infty} \leq \dd$, and then we have $\norm{f^{k,w}_{z}-f^{k,w}_{p_{i,j}}}_{L_{2}(Q)} \leq \Bdd(w;\cD) \Lip(k)\sqrt{\Bdd(k)} \dd$ from Equation \eqref{eq:donsker_1}.
Let $\ee=\Lip(k)\Bdd(k)^{-1/2}\dd$, then $\ee \norm{E}_{L_{2}(Q)}=\Bdd(w;\cD) \Lip(k)\sqrt{\Bdd(k)}\dd$.
Since $\cF^{k,w}_{\{p_{i,j}\}_{i,j=1}^{N}}$ covers $\cF^{k,w}_{A}$ by $L_{2}(Q)$-balls with radius $\ee \norm{E}_{L_{2}(Q)}$, the covering number $N(\ee\norm{F}_{L_{2}(Q)}, \cF^{k,w}_{A}, L_{2}(Q))$ is bounded from above by $N^{2}=(T / \dd)^{2} = (T \Lip(k)\Bdd(k)^{-1/2} \ee^{-1})^{2}$.
Now, we have\footnote{Without loss of generality, we can make $\log(T \Lip(k)\Bdd(k)^{-1/2})>0$ by retaking larger $T$ or $\Lip(k)$ if necessary. 
Remark that $- \log \ee \geq 0$ for $\ee \in (0,1]$ and $ \int_{0}^{1} \sqrt{- \log \ee }d\ee < \infty$.}
\begin{align*}
&\int_{0}^{1} \sqrt{\log \sup_{Q} N(\ee\norm{F}_{L_{2}(Q)}, \cF^{k,w}_{A}, L_{2}(Q))}d\ee \\
&\leq \int_{0}^{1} \sqrt{\log(T \Lip(k)\Bdd(k)^{-1/2} \ee^{-1} )^{2}}d\ee \\
& \leq \sqrt{2}\int_{0}^{1} \sqrt{\log(T \Lip(k)\Bdd(k)^{-1/2}) }d\ee+\sqrt{2} \int_{0}^{1} \sqrt{- \log \ee }d\ee \\
&<\infty.
\end{align*}
Let $z \in A$ be any position and $\{z_{n} \in A \cap \lQ^{2}\}_{n \in \lN}$ be a sequence satisfying $\lim_{n\to \infty}z_{n} = z$.
Since $A$ is dense in $\lR^{2}$, such a sequence exists.
Since the restriction $\cF^{k,w}_{A \cap \lQ^{2}}$ is countable and $\lim_{n\to \infty}f^{k,w}_{z_{n}}(D) = f^{k,w}_{z}(D)$ for any $D \in \cD$ from Equation \eqref{eq:conti_donsker}, $\cF^{k,w}_{A}$ is pointwise measurable.
Therefore, by Proposition \ref{prop:donsker_condition}, $\cF^{k,w}_{A}$ is shown to be $P$-Donsker.
\qed
\end{proof}
To sum up, the uniform confidence band for $\{ P f^{k,w}_{z}\}_{z \in A}$ is obtained as follows:
\begin{theorem}
\label{thm:confi_interval_pwk}
Let $k \in \Kc$, $w$ be a weight function for $\cD$, $P$ be a Radon probability measure on $\cD$, $D_{1},\ldots, D_{n}, \iid \sim P$, and $A$ is a bounded dense subset in $\lR^{2}$.
Then, for almost all $\alpha \in (0,1)$, we have
\[
\lim_{n\to\infty}\prob\pare{ P f^{k,w}_{z} \in \bracket{ \lP_{n}f^{k,w}_{z}-\frac{\hat{\xi}^{*}_{n, \alpha}}{\sqrt{n}}, \lP_{n}f^{k,w}_{z}+\frac{\hat{\xi}^{*}_{n, \alpha}}{\sqrt{n}}} \mbox{ {\rm  for any} } z \in A} = 1-\alpha,
\]
where $\hat{\xi}^{*}_{n, \alpha}(\omega)$ is the upper $\alpha$-quantile of $\sup_{z \in A} \abs{ \lG_{n}^{*}f^{k,w}_{z}(\omega)}$ for $\omega \in \Omega$.
\end{theorem}

A computation algorithm of the upper $\alpha$-quantile $\hat{\xi}^{*}_{n, \alpha}(\omega)$ in Theorem \ref{thm:confi_interval_pwk} is given as follows:
\begin{center}
\begin{minipage}{0.9\linewidth}
\begin{algorithm}[H]
\caption{The bootstrap algorithm for the expectation of the PWK vector}
\label{alg:bootstrap_pwgk}
\begin{algorithmic}[1]
\REQUIRE $k \in \Kc$. $w$: a weight function for $\cD$. $A$: a bounded dense subset in $\lR^{2}$. $\alpha \in (0, 1)$. $b \in \lN$. $D_{1},\ldots,D_{n}$:  realizations of i.i.d. samples drawn from $P$. 
\STATE Compute $\lP_{n}f^{k,w}_{z}:=n^{-1}\sum_{i=1}^{n}V^{k,w}(D_{i})(z)$.
\FOR{$\ell = 1,\ldots,b$}
\STATE Resample $n$ persistence diagrams $D^{*}_{1, \ell},\ldots,D^{*}_{n, \ell}$ from $\{D_{1},\ldots,D_{n}\}$. 
\STATE Compute $\lP^{*}_{n, \ell}f^{k,w}_{z}:=n^{-1} \sum_{i=1}^{n}V^{k,w}(D^{*}_{i, \ell})(z)$.
\STATE Compute $\hat{\eta}^{*}_{n,\ell}=\sup_{z \in A} \abs{\sqrt{n} \pare{ \lP^{*}_{n, \ell}f^{k,w}_{z}-\lP_{n} f^{k,w}_{z}}}$.
\ENDFOR
\STATE Compute the empirical distribution as 
\[
\tilde{\lF}_{n,b}(c):=\frac{1}{b} \sum_{\ell=1}^{b} \one_{[\hat{\eta}^{*}_{n,\ell},\infty) } (c),
\]
where $\one_{[s,t)}:\lR \to \{0, 1\}$ is the indicator function on $[s,t)$.
\RETURN $\tilde{\xi}_{n, \alpha, b} := \inf \braces{ c \remiddle \tilde{\lF}_{n,b}(c) \geq 1- \alpha }$.
\end{algorithmic}
\end{algorithm}
\end{minipage}
\end{center}

\section{Stability of the expectation of the PWK vector}
\label{sec:stability_expectation}
Let $m \in \lN$, $(M, d_{M})$ be a metric space, and $M^{m}=M \times \cdots \times M$ be the product metric space with the Hausdorff distance.
By regarding $M^{m}$ as a measurable set by its Borel $\sigma$-set, we can define a probability measure $\cP_{m}$ on $M^{m}$ and draw a sample $X \sim \cP_{m}$.
While a realization of the sample $X=(x_{1},\ldots,x_{m})$ is an element $M^{m}$, we can also view $X$ as an $m$-points set in $M$.
For another probability measure $\cQ_{m}$ on $M^{m}$ and a sample $Y \sim \cQ_{m}$, it is shown from Theorem \ref{thm:stability_ball} that
\begin{align}
\dW{\infty}(D_{q}(\lB(X(\omega))),D_{q}(\lB(Y(\omega))) \leq \dH(X(\omega),Y(\omega)) \mbox{ for any } \omega \in \Omega. \label{eq:stab_point_prob}
\end{align}
However, even if $\cP_{m}=\cQ_{m}$, there can exist $\omega \in \Omega$ such that $\dH(X(\omega),Y(\omega))$ takes a large value.
In such a situation, \cite{CFLMRW15} estimate Equation \eqref{eq:stab_point_prob} in a probabilistic sense by the Wasserstein distance for probability measures.
The {\em $p$-Wasserstein distance} $(1 \leq p < \infty)$ between probability measures $\mu$ and $\nu$ on $M$ is defined by
\begin{align}
W_{p}[d_{M}](\mu,\nu):= \inf_{\pi \in \Pi(\mu,\nu)} \pare{ \int_{M} d_{M}(x,y)^{p} d\pi(x,y)}^{\frac{1}{p}}, \label{eq:wasserstein_prob}
\end{align}
where $\Pi(\mu,\nu)$ is the set of couplings between $\mu$ and $\nu$.\footnote{A probability measure $\pi$ on $M \times M$ is called a coupling between $\mu$ and $\nu$ if, for a natural projection $p_{i}(x_{1},x_{2})=x_{i} ~ (i=1,2, ~ x_{i} \in M)$, their induced measures satisfy $(p_{1})_{*}\pi=\mu$ and $(p_{2})_{*}\pi=\nu$. }
Let $\phi:(\cF(M), \dH) \to (\cD(\lB(M)), \dW{\infty}), X \mapsto D_{q}(\lB(X))$ be the map from a point set to the persistence diagram, where $\cF(M)$ is the set of finite point sets in $M$, and $\phi_{*}\cP_{m}$ denote its induced measure of $\cP_{m}$.
Then, the stability theorem for point sets (Theorem \ref{thm:stability_ball}) is generalized as follows:
\begin{proposition}[Lemma 16 in the Supplementary Material of \cite{CFLMRW15}]
\label{prop:stab_prob_pd}
Let $(M,d_{M})$ be a metric space, $\cP_{m}$ and $\cQ_{m}$ be probability measures on $M^{m}$, and $1 \leq p < \infty$.
Then, we have
\begin{align}
W_{p}[\dW{\infty}](\phi_{*} \cP_{m}, \phi_{*} \cQ_{m}) \leq W_{p}[\dH](\cP_{m}, \cQ_{m}). \label{eq:stab_prob_pd}
\end{align}
\end{proposition}

In addition, \cite{CFLMRW15} estimate the right hand side in Equation \eqref{eq:stab_prob_pd} for a specific class of probability measures on $M^{m}$.
Let $\mu$ and $\nu$ be probability measures on $M$.
Then, the product measures $\mu^{\otimes m}$ and $\nu^{\otimes m}$ are probability measures on $M^{m}$ and a sample $X=(x_{1},\ldots,x_{m}) \sim \mu^{\otimes m}$ can be seen as a set composed of samples $x_{1},\ldots,x_{m}, \iid \sim \mu$.
\begin{proposition}[Lemma 15 in the Supplementary Material of \cite{CFLMRW15}]
\label{prop:stab_prob_single}
Let $(M,d_{M})$ be a metric space, $\mu$ and $\nu$ be probability measures on $M$, and $1 \leq p < \infty$.
Then, we have
\begin{align*}
W_{p}[\dH](\mu^{\otimes m}, \nu^{\otimes m}) \leq m^{\frac{1}{p}} W_{p}[d_{M}](\mu, \nu).
\end{align*}
\end{proposition}

In Section \ref{subsec:lattice}, we will apply a PWK vector to the perturbed lattice system.
Then, the class of a product measure composed of a single probability measure, like $\mu^{\otimes m}$, will turn out to be not suitable to analyze perturbed lattices.
\begin{example}
\label{exam:lattice}
Let $L:=\{(i,j) \in \lR^{2} \mid i,j=1,\ldots,m\}$ be a $2$-dimensional square lattice with finite size $m$ and $\mu$ be a probability measure on $\lR^{2}$ with mean zero.
A {\em perturbed lattice} of $L$ by $\mu$ is defined by $L_{\mu}:=\{x + e_{x} \mid x \in L \}$ where $\{e_{x}\}_{x \in L}$ is the set of i.i.d. samples drawn from $\mu$.
Then, the perturbed lattice $L_{\mu}$ is a sample drawn from $\otimes_{x \in L} \mu_{x}$ where $\mu_{x}=(s_{x})_{*}\mu$ is the induced measure of $\mu$ by a translation $s_{x}:\lR^{2} \to \lR^{2}, z \mapsto z + x$.
\end{example}

For a probability measure which is not composed of a single probability measures, like $\otimes_{x \in L} \mu_{x}$ in Example \ref{exam:lattice}, we will show a generalized stability theorem.
Let $\{\mu_{i}\}_{i=1}^{m}$ be a family of probability measures on $M$.
Then, the product measure $\otimes_{i=1}^{m}\mu_{i}$ is a probability measure on $M^{m}$, and a sample $X \sim \otimes_{i=1}^{m}\mu_{i}$ is interpreted as a point set $X=\{x_{1},\ldots,x_{m} \}$ composed of a sample $x_{i} \sim \mu_{i}$ for $i=1,\ldots,m$.
\begin{proposition}
\label{prop:stab_prob_multi}
Let $(M,d_{M})$ be a metric space, $\{\mu_{i}\}_{i=1}^{m}$ and $\{\nu_{i}\}_{i=1}^{m}$ be families of probability measures on $M$, and $1 \leq p < \infty$.
Then, we have
\begin{align*}
W_{p}[\dH](\otimes_{i=1}^{m}\mu_{i}, \otimes_{i=1}^{m}\nu_{i}) \leq \pare{\sum_{i=1}^{m} W_{p}[d_{M}](\mu_{i}, \nu_{i})^{p} }^{1/p}.
\end{align*}
\end{proposition}

\begin{proof}
Note that, for point sets $X=\{x_{1},\ldots,x_{m} \}$ and $Y=\{y_{1},\ldots,y_{m} \}$ in $M$, we have
\begin{align}
\dH(X,Y)^{p} 
&= \max_{i} \min_{j} d_{M}(x_{i},y_{j})^{p} \nonumber  \\
&\leq \sum_{i=1}^{m} \min_{j} d_{M}(x_{i},y_{j})^{p} ~~ (\norm{v}_{\infty} \leq \norm{v}_{p} \mbox{ for any $p\geq 1$ and $v \in \lR^{m}$}) \nonumber \\
&\leq \sum_{i=1}^{m} d_{M}(x_{i},y_{i})^{p}. \label{eq:hausdorff_lp}
\end{align}
Let $\pi_{i} \in \Pi(\mu_{i},\nu_{i}) ~ (i=1,\ldots,m)$ be any coupling.
Then, $\otimes_{i=1}^{m} \pi_{i} \in \Pi(\otimes_{i=1}^{m}\mu_{i}, \otimes_{i=1}^{m}\nu_{i})$ and we have
\begin{align*}
W_{p}[\dH](\otimes_{i=1}^{m}\mu_{i}, \otimes_{i=1}^{m}\nu_{i})^{p} 
&=\inf_{\pi \in \Pi(\otimes_{i=1}^{m}\mu_{i}, \otimes_{i=1}^{m}\nu_{i})} \int_{M^{m} \times M^{m}} \dH(X,Y)^{p} d\pi(X, Y) \\
&\leq \int_{M^{m} \times M^{m}} \dH(X,Y)^{p} d(\otimes_{i=1}^{m}\pi_{i})(X, Y) \\
&\leq \sum_{i=1}^{m} \int_{M \times M} d_{M}(x_{i},y_{i})^{p} d\pi_{i}(x_{i},y_{i}) ~~ (\mbox{Equation \eqref{eq:hausdorff_lp}}).
\end{align*}
By taking the infimum for each $\pi_{i}$, we have
\[
W_{p}[\dH](\otimes_{i=1}^{m}\mu_{i}, \otimes_{i=1}^{m}\nu_{i})^{p} \leq \sum_{i=1}^{m} W_{p}[d_{M}](\mu_{i},\nu_{i})^{p}.
\]
\qed
\end{proof}
Obviously, Proposition \ref{prop:stab_prob_single} is contained in Proposition \ref{prop:stab_prob_multi} by setting $\mu_{i} \equiv \mu$ and $\nu_{i} \equiv \nu$ for all $i=1,\ldots,m$.
From Theorem \ref{thm:kernel_stability}, the map $V^{k,w}:\cD \to \cH_{k}$ is shown to be Lipschitz continuous.
The following theorem is its generalization to probability distributions:
\begin{theorem}
\label{thm:stab_prob_pwk}
Let $P$ and $Q$ be Radon probability measures on $\cD$, $1 \leq p < \infty$, $k \in \Kc$, and $w\in \cW_{\infty}(\cD)$.
Then, we have
\begin{align*}
&\norm{\lE_{D \sim P}[V^{k,w}(D)] - \lE_{E \sim Q}[V^{k,w}(E)]}_{\cH_{k}} \\
&\leq (\Lip(k)\Bdd(w;\cD)+\sqrt{\Bdd(k)}L_{\infty}(w;\cD))W_{p}[\dW{\infty} ](P, Q).
\end{align*}
\end{theorem}

\begin{proof}
Let $\pi \in \Pi(P,Q)$ be any coupling, then we have
\begin{align*}
&\norm{\lE_{D \sim P}[V^{k,w}(D)] - \lE_{E \sim Q}[V^{k,w}(E)]}_{\cH_{k}}^{p} \\
&=\norm{ \int_{\cD}V^{k,w}(D)dP(D) -  \int_{\cD}V^{k,w}(E)dQ(E)}_{\cH_{k}}^{p} \\
&=\norm{ \int_{\cD \times \cD} \pare{ V^{k,w}(D) - V^{k,w}(E) }d \pi(D, E)}_{\cH_{k}}^{p} ~~ (\mbox{Definition of a coupling } \pi)\\
&\leq \pare{ \int_{\cD \times \cD} \norm{ V^{k,w}(D) - V^{k,w}(E) }_{\cH_{k}} d \pi(D, E)}^{p} ~~ (\| \lE[V] \|_{\cH_{k}} \leq \lE[\| V \|_{\cH_{k}}])\\
&\leq \int_{\cD \times \cD} \norm{ V^{k,w}(D) - V^{k,w}(E)}_{\cH_{k}}^{p}d \pi(D, E) ~~ (\mbox{Jensen's inequality}) \\
&\leq \int_{\cD \times \cD} (\Lip(k)\Bdd(w;\cD)+\sqrt{\Bdd(k)}L_{\infty}(w;\cD))^{p} \dW{\infty}(D,E)^{p}d \pi(D, E) ~~ (\mbox{Theorem \ref{thm:kernel_stability}}) \\
&= (\Lip(k)\Bdd(w;\cD)+\sqrt{\Bdd(k)}L_{\infty}(w;\cD))^{p}W_{p}[\dW{\infty} ](P, Q)^{p}.
\end{align*}
\qed
\end{proof}

Combining Proposition \ref{prop:stab_prob_pd}, Proposition \ref{prop:stab_prob_multi}, and Theorem \ref{thm:stab_prob_pwk}, the difference between two expectations of PWK vectors is estimated by the differences between probability distributions on $M$.
\begin{corollary}
\label{cor:stability_dist}
Let $(M,d_{M})$ be a metric space, $\{\mu_{i}\}_{i=1}^{m}$ and $\{\nu_{i}\}_{i=1}^{m}$ be families of Radon probability measures on $M$, $1 \leq p < \infty$, $k \in \Kc$, and $w \in \cW_{\infty}(\cD(\lB(M)))$.
Then, we have
\begin{align*}
&\norm{\lE_{X \sim \otimes_{i=1}^{m}\mu_{i}}[V^{k,w}(D_{q}(\lB(X)))]- \lE_{Y \sim \otimes_{i=1}^{m}\nu_{i} }[V^{k,w}(D_{q}(\lB(Y)))]}_{\cH_{k}} \\
&\leq (\Lip(k)\Bdd(w;\cD(\lB(M)))+\sqrt{\Bdd(k)}L_{\infty}(w;\cD(\lB(M))))\pare{\sum_{i=1}^{m} W_{p}[d_{M}](\mu_{i}, \nu_{i})^{p} }^{1/p}.
\end{align*}
\end{corollary}

\section{Numerical experiments}
\label{sec:numarical}
In this section, we apply the PWK vector to probability distributions and compare the performance among the PWK vector and other statistical methods for persistence diagrams.
All persistence diagrams are obtained from the ball model filtration, in dimension $1$, and computed by {\rm HomCloud}.\footnote{\url{http://www.wpi-aimr.tohoku.ac.jp/hiraoka_labo/homcloud/index.en.html}.}

Let $D_{1},\ldots,D_{n}$ be persistence diagrams with some property (A) and $E_{1},\ldots,E_{n}$ be persistence diagrams with some property (B).
Then, we consider to decide whether the properties (A) and (B) are the same or not from the sample sets $\{D_{i}\}_{i=1}^{n}$ and $\{E_{j}\}_{j=1}^{n}$.
This problem is called the {\em two sample problem} and we formalize this problem in terms of the hypothesis testing as follows:
Let $P$ and $Q$ be probability measures on $\cD$, then we define a null hypothesis by $H_{0}:P=Q$ and the alternative hypothesis by $H_{1}:P \neq Q$.
If $P=Q$, it is expected that $\norm{n^{-1}\sum_{i=1}^{n}V^{k,w}(D_{i}) - n^{-1}\sum_{i=1}^{n}V^{k,w}(E_{i}) }_{\cH_{k}}$ is small from Theorem \ref{thm:SLNN_pwk} and Theorem \ref{thm:stab_prob_pwk}, however, the problem is how small the threshold determines whether the null hypothesis $H_{0}$ is accepted or rejected.
To solve this, we apply the kernel method to the two sample problem \cite{GBRSS07,GBRSS12} (see also Appendix \ref{sec:ktst}) for persistence diagrams.
In order to measure the performance of the hypothesis testing procedure, we compute the type I and type II error for the testing, where the {\em type {\rm I} error} is defined by the probability that $H_{0}$ is rejected when $H_{0}$ is true and the {\em type {\rm II} error} is defined by the probability that $H_{0}$ is accepted when $H_{1}$ is true.
If a hypothesis testing procedure decides that any null hypothesis is rejected, then the type I error is $0$ but the type II error is $1$.
Hence, it is desired that both type I and type II error of a hypothesis testing procedure are small.

In the kernel two sample test, we use the following statistics
\begin{align*}
\mathrm{MMD}_{u}(\bm{D}_{n},\bm{E}_{n};K)^{2}:=& \frac{1}{n(n-1)}\sum_{i =1}^{n} \sum_{j \neq i}^{n} K(D_{i},D_{j}) \\
&+ \frac{1}{n(n-1)}\sum_{a=1}^{n}\sum_{b \neq a}^{n} K(E_{a},E_{b}) - \frac{2}{n^{2}}\sum_{i=1}^{n} \sum_{a=1}^{n}K(D_{i},E_{a}),
\end{align*}
where $K$ is a positive definite kernels $K$ on $\cD$, and an algorithm to compute the type I error is given as follows (see also Appendix \ref{sec:ktst}):
\begin{center} 
\begin{minipage}{0.9\linewidth}
\begin{algorithm}[H]
\caption{The algorithm for the type I error by the kernel method for persistence diagrams}
\label{alg:type_error_pd}
\begin{algorithmic}[1]
\REQUIRE $\alpha \in (0, 1)$. $q,m,N \in \lN$. $D_{1},\ldots,D_{n}$: i.i.d. samples drawn from $P$. $E_{1},\ldots,E_{n}$: i.i.d. samples drawn from $Q$. $K$: a positive definite kernel for persistence diagrams.
\STATE Compute the upper $\alpha$-quantile $\hat{\xi}_{n,1-\alpha}$ of $n\mathrm{MMD}_{u}(\bm{D}_{n}, \bm{E}_{n};K)^{2}$ by the bootstrap method on the aggregated data.
\FOR{$\ell=1,\ldots,N$}
\STATE Resample $m$ samples $D_{1,\ell},\ldots,D_{m,\ell}$ (resp. $E_{1,\ell},\ldots,E_{m,\ell}$) from $\{D_{1},\ldots,D_{n}\}$ (resp. $\{E_{1},\ldots,E_{n}\}$). 
\STATE Compute $m\mathrm{MMD}_{u}(\bm{D}_{m,\ell}, \bm{E}_{m,\ell};K)$.
\ENDFOR
\STATE Compute $\hat{p}:=N^{-1} \sum_{\ell=1}^{N} \one_{(m\mathrm{MMD}_{u}(\bm{D}_{m,\ell}, \bm{E}_{m,\ell};K) \leq \hat{\xi}_{n,1-\alpha})}$.
\RETURN $1-\hat{p}$. 
\end{algorithmic}
\end{algorithm}
\end{minipage}
\end{center}
The output of Algorithm \ref{alg:type_error_pd} is the type I error under $H_{0}$.
With respect to the type II error, we consider another null distribution $H_{0}^{*}:P \neq Q$ and set $\hat{p}$ as the type II error under $H_{0}$ where $1-\hat{p}$ is the type I error under $H_{0}^{*}$.

We remark that, when we apply Algorithm \ref{alg:type_error_pd} to $X_{1},\ldots,X_{n}, \iid \sim \cP$ and $Y_{1},\ldots,Y_{n}, \iid \sim \cQ$ where $\cP$ and $\cQ$ are  probability measures on the set $\cF(\lR^{d})$ of finite point sets in $\lR^{d}$, the null hypothesis $H_{0}:\phi_{*}\cP=\phi_{*}\cQ$ is also denoted by $H_{0}:\cP=\cQ$ in this section.

\subsection{Kernels on persistence diagrams}
In order to apply the kernel two sample test to persistence diagrams, we will define four positive definite kernels on persistence diagrams $\cD$: the PWK vector, a persistence landscape, the persistence scale-space kernel, and the Sliced Wasserstein kernel.\footnote{For comparison in machine learning tasks among the PWK vector, a persistence landscape, and the persistence scale-space kernel, please see \cite{KFH18}.}

\subsubsection{PWK vector}
We define a positive kernel of the PWK as the Gaussian kernel on the RKHS:
\begin{align}
K^{{\rm PW}}(D,E;k,w) =\exp\pare{-\frac{\norm{V^{k,w}(D)-V^{k,w}(E)}^{2}_{\cH_{k}}}{2 \tau^{2}}} \label{kernel_pwgk},
\end{align}
where $\tau$ is a positive parameter.
We fix a positive definite kernel on $\lR^{2}$ as the Gaussian kernel $k(x,y)=e^{-\frac{\norm{x-y}^{2}}{2\sigma^{2}}}$ and select a weight function from the unweighted function $w_{0}(x)\equiv 1$, the linear weight function $w_{1}(x)=\pers(x)$, and the arctangent weight function $w_{{\rm arc}}(x)=\arctan(C_{{\rm arc}} \pers(x)^{5})$.
All parameters $\sigma, C_{{\rm arc}}, \tau$ are computed as follows:\footnote{For the reason of fixing the parameters in this way, please also see \cite{KFH18}.}
For a collection of persistence diagrams $\cD=\{D_{\ell} \}_{\ell=1}^{n}$, we set 
\begin{align}
\sigma &=\median \{ \sigma(D_{\ell}) \}_{\ell=1}^{n}, ~ \sigma(D)=\median \{ \norm{x_{i}-x_{j}} \mid x_{i},x_{j} \in D, \ i<j \}, \label{eq:sigma} \\
C&=( \median \{\pers(D_{\ell})  \}_{\ell=1}^{n} )^{-p}, ~  \pers(D)=\median \{ \pers(x_{i}) \mid x_{i} \in D \}, \nonumber \\
\tau &= \median \braces{ \norm{V^{k,w}(D_{i})-V^{k,w}(D_{j})}_{\cH_{k}}  \remiddle 1 \leq i < j \leq n}. \label{eq:tau}
\end{align}

\subsubsection{Persistence landscape}
A {\em persistence landscape} \cite{Bu15} is defined by
\[
V^{{\rm PL}}(D) (k,t) := k \mbox{-th largest value of } \max \{ \min \{ t-b_{i},d_{i}-t\} , 0 \} ~~ (k \in \lN, ~ t \in \lR) \label{eq:landscape}
\]
and $V^{{\rm PL}}(D)$ is in $L^{2}(\lN \times \lR)$ when $\card{D}$ is finite.
From the construction, a generator with large (resp. small) persistence has a large (resp. small) value in the persistence landscape.
For example, the persistence landscape $V^{{\rm PL}}(D)(1,t)$ of a persistence diagram $D=\{(b,d)\}$ of a single generator has a peak at $(b+d)/2$ whose height is $(d-b)/2$, which is linear in the persistence.
By using the inner product of $L^{2}(\lN \times \lR)$, we consider the linear kernel on $L^{2}(\lN \times \lR)$:
\begin{align}
K^{{\rm PL}}(D,E) =\inn{V^{{\rm PL}}(D)}{V^{{\rm PL}}(E)}_{L^{2}(\lN \times \lR)}. \label{eq:kernel_pl}
\end{align}

\subsubsection{Persistence scale-space kernel}
The {\em persistence scale-space kernel} (PSSK) \cite{RHBK15} is defined by
\begin{align}
K^{{\rm PSS}}(D,E;t)=\frac{1}{8 \pi t} \sum_{x \in D} \sum_{y \in E}  e^{-\frac{ \norm{x-y}^{2} }{8 t}} - e^{-\frac{ \norm{x-\bar{y}}^{2} }{8 t}}, \label{eq:pssk}
\end{align}
where $\bar{y}:=(y^{2},y^{1})$ for $y=(y^{1},y^{2})$.
The vector representation of $D$ by the PSSK is $\Phi_{t}(D)(x)=\frac{1}{4\pi t} \sum_{y \in D} e^{-\frac{\norm{x-y}^{2}}{4 t}} - e^{-\frac{\norm{x-\bar{y}}^{2}}{4 t}}$, which is motivated from the solution of the heat diffusion equation for $D$, and the PSSK is given by the $L^{2}$-inner product, i.e., $K^{{\rm PSS}}(D,E;t)=\inn{\Phi_{t}(D)}{\Phi_{t}(E)}_{L^{2}(\ad)}$.
We set $t=\sigma^{2} / 2$, where $\sigma$ is of Equation \eqref{eq:sigma}, so as the kernel bandwidths in $V^{\kG,w}(D)$ and $\Phi_{t}(D)$ are the same.

As a modification of the PSSK to improve treatments of probability distributions on $\cD$, the {\em universal persistence scale-space kernel} ({\em u-PSSK}) \cite{KHNLB15} is proposed as follows:
\[
K^{{\rm PSS}}_{{\rm u}}(D,E;t):=\exp \pare{ \norm{V^{{\rm PSS}}_{t}(D)-V^{{\rm PSS}}_{t}(E)}_{L^{2}(\ad)}^{2}}.
\]

\subsubsection{Sliced Wasserstein kernel}
Another positive definite kernel for persistence diagrams uses an idea of the sliced Wasserstein distance.
For $\theta \in S^{1}$, $L(\theta)$ denotes the line $\{\lambda (\cos \theta, \sin \theta) \mid \theta \in \lR\}$ and $\pi_{\theta}:\lR^{2} \to L(\theta)$ denotes the orthogonal projection onto the line $L(\theta)$.
Let $p_{\DD}:\lR^{2} \to \DD$ be an orthogonal projection onto the diagonal.
Here, for finite nonnegative measures $\mu$ and $\nu$ on $\lR$ with the same mass, i.e., $\mu(\lR)=\nu(\lR)$, the generalized $1$-Wasserstein distance is denoted by
\[
\cW(\mu,\nu):=m \inf_{\pi \in \Pi(\mu',\nu')}\int_{\lR \times \lR} \abs{x-y} d\pi(x,y),
\]
where $m=\mu(\lR)$, $\mu'=\mu/m$, and $\nu'=\nu / m$.
Let $D$ and $E$ be finite persistence diagrams.
Then, the {\em Sliced Wasserstein distance} between $D$ and $E$ is defined by
\[
{\rm SW}(D,E):=\frac{1}{2\pi} \int_{S^{1}} \cW ( \sum_{x \in D} \dd_{\pi_{\theta}(x)}+ \sum_{y \in E} \dd_{\pi_{\theta} \circ p_{\DD} (y)},  \sum_{y \in E} \dd_{\pi_{\theta}(y)}+ \sum_{x \in D} \dd_{\pi_{\theta} \circ p_{\DD} (x)} )d\theta.
\]
The induced positive definite kernel from the Sliced Wasserstein distance
\[
K^{{\rm SW}}(D,E):=\exp \pare{ - \frac{{\rm SW} (D,E)}{2\tau^{2}}}
\]
is called the {\em Sliced Wasserstein kernel} \cite{CCO17}.

\subsection{Perturbed lattice}
\label{subsec:lattice}
We begin with the kernel two sample test for persistence diagrams of synthesized data.
Let $L:=\{(i,j) \mid i,j=1,\ldots,m_{L}\}$ be a $2$-dimensional square lattice with size $m_{L} \in \lN$.
As mentioned in Example \ref{exam:lattice}, a sample $X$ drawn from the product measure $\cP^{r}:=\bigotimes_{x \in L} \mathrm{Unif}({\rm Box}(x,r))$ is seen as a point set representing a perturbed lattice, where ${\rm Box}(x,r)=[x_{1} - r, x_{1} + r] \times [x_{2} - r, x_{2} + r]$ is a square region centered at $x=(x_{1},x_{2}) \in \lR^{2}$ with parameter $r>0$.
We also consider another probability measure $\cQ^{s}=\bigotimes_{x \in L} \cN_{2}(x,s^{2}I_{2})$ with parameter $s>0$.

We remark that, for $U_{1},U_{2}, \iid \sim \mathrm{Unif}([-r,r]) ~ (r>0)$, we have $(U_{1},U_{2}) \sim \mathrm{Unif}({\rm Box}(0,r))$ and the covariance matrix of $(U_{1},U_{2})$ is given by $(r^{2}/3) I_{2}$ since the variance of the uniform distribution on $[a,b] \subset \lR$ is $(b-a)^{2}/12$.
Thus, for any $x \in \lR^{2}$, $\mathrm{Unif}({\rm Box}(x,r))$ and $\cN_{2}(x,s^{2}I_{2})$ have the same mean and covariance matrix when $r^{2}=3s^{2}$.

In our experiment, we fix $s=0.1$ and consider the two sample problem where the null hypothesis is $H^{r}_{0}:\cP^{r}=\cQ^{0.1}$ and the alternative hypothesis is $H^{r}_{1}:\cP^{r} \neq \cQ^{0.1}$.
Since $\cP^{r}$ and $\cQ^{0.1}$ are not the same for any $r$, the null hypothesis $H^{r}_{0}$ should be rejected, and hence the type I error should be small.
We consider the $2$-dimensional square lattice $L$ with size $m_{L}=20$, take $n=50$ samples from each distribution $\cP^{r}$ and $\cQ^{0.1}$, and fix the parameters in Algorithm \ref{alg:type_error_pd} by $\alpha=0.01, q=1, m=40, N=1000$.\footnote{$\alpha$ is the significance level. $q$ is the dimension of persistence diagrams. $m$ is the number of sampling to compute the type I error empirically. $N$ is the number of trials to compute the type I error.}
Table \ref{table:test_lattice} shows the type I and type II errors of the kernel two sample test under the null hypothesis $H^{r}_{0}$.
\begin{table}[htbp]
\begin{center}
\caption{Each value in the table is the type I error (left) and the type II error (right).
$\star$ (resp. $\blacktriangle$) means the results of the PWK vector and the PSSK with the larger (resp. smaller) bandwidth in the Gaussian kernel on $\lR^{2}$ which is given by $10\sigma$ and $10t$ (resp. $\sigma/10$ and $t/10$). }
\label{table:test_lattice}
\begin{tabular}{|c|c||c|c|c||c|c|} \hline
\multicolumn{2}{|c||}{The null hypothesis}   &  $H^{0.1\sqrt{2}}$    & $H^{0.1\sqrt{3}}$ & $H^{0.2}$ 	& $\star$ $H^{0.1\sqrt{3}}$  & $\blacktriangle$ $H^{0.1\sqrt{3}}$     \\ \hline\hline
		     	        & $w_{{\rm 0}}$           &0.00/0.00		&0.00/0.00  	&0.00/0.00  	& 0.00/1.00   &0.00/0.01 \\ \cline{2-7}
PWK		       	        & $w_{{\rm 1}}$           &0.00/0.00		&0.00/0.00   	&0.00/0.00  	& 0.00/0.01   &0.00/0.00 \\ \cline{2-7}
    				& $w_{{\rm arc}}$        &0.00/0.00		&0.00/0.00   	&0.00/0.00  	& 0.00/0.45   &0.00/0.00 \\ \cline{1-7}
\multicolumn{2}{|c||}{PSSK}			  &0.00/0.00		&0.00/0.02  	&0.00/0.00 	& 0.00/0.03   &0.00/0.02 \\ \hline 
\multicolumn{2}{|c||}{universal PSSK}	 &0.00/0.00		&0.00/0.02  	&0.00/0.00 	& 0.00/0.02   &0.00/0.06 \\ \hline 
\multicolumn{2}{|c||}{Persistence landscape}&0.00/0.00	 	&0.00/0.91        &0.00/0.03  	& --  & --  \\ \hline     
\multicolumn{2}{|c||}{Sliced Wasserstein}   &0.00/0.00	  	&0.00/0.00   	&0.00/0.00  	& --  & --  \\ \hline     
\end{tabular}
\end{center}
\end{table}

With respect to $H^{0.1\sqrt{2}}_{0}$ and $H^{0.2}_{0}$, all methods correctly show low error rates.
The problem is $H^{0.1\sqrt{3}}_{0}$ because $\mathrm{Unif}({\rm Box}(x,0.1\sqrt{3}))$ and $\cN_{2}(x,(0.1\sqrt{3})^{2}I_{2})$ have the same mean and covariance matrix for any $x \in \lR^{2}$.
The PWK vector, the (universal) PSSK, and the Sliced Wasserstein kernel can detect the difference between $\cP^{0.1\sqrt{3}}$ and $\cQ^{0.1}$ correctly.
The result of a persistence landscape, showing the high type I error, is considered to be caused by its construction: a value in the persistence landscape is linearly dependent on the persistence.
On purpose, we try the testing with larger and smaller parameters $\sigma$ in $K^{{\rm PW}}$ and $t$ in $K^{{\rm PSS}}$, then the results get worse, which implies the selection of parameters is an important issue as also mentioned in \cite{KFH18}.

In order to focus on differences between the two distributions $\cP^{0.1\sqrt{3}}$ and $\cQ^{0.1}$, we will compare confidence intervals for the expectation of the PWK vector.
A uniform confidence band $\tilde{\xi}_{n,\alpha,b}$ is computed from Algorithm \ref{alg:bootstrap_pwgk}, where the parameters are fixed as $\alpha=0.05$ and $b=10^{4}$. 
For $x=(x_{1},x_{2}) \in \lR^{2}$ and $r>0$, we fix $\mathrm{Box}(x;r)$ as a bounded dense subset of $\lR^{2}$ and consider the set of confidence intervals for $\{\phi_{*}\cP f^{k,w}_{z} \mid z \in I_{0}^{20}(x,r)\}$ where $I_{s}^{t}(x,r)=\{(x_{1}, x_{2} + r(0.1i-1))\}_{i=s}^{t}$ $ (s < t \in \lN \cup \{0\})$ as an index set.\footnote{Note that $\phi_{*}\cP f^{k,w}_{z}$ is the value at $z$ of the expectation of the PWK vector, that is, $\phi_{*}\cP f^{k,w}_{z}=\lE_{X \sim \cP}[V^{k,w}(D_{q}(\lB(X)))](z)$.}
Since all generators in dimension $1$ of the $2$-dimensional square lattice $L$ appear at $1/2$ and disappear $\sqrt{2}/2$, we fix $x=(1/2,\sqrt{2}/2)$ and try $r=0.05$.
Then, the confidence intervals for $\phi_{*}\cP^{0.1\sqrt{3}} f^{k,w}_{z}$ and $\phi_{*}\cQ^{0.1} f^{k,w}_{z}$ do not have an intersection on $z \in I_{6}^{17}(x,r)$ for $w_{0}$, $z \in I_{0}^{20}(x,r)$ for $w_{1}$, and $z \in I_{8}^{11}(x,r)$ for $w_{{\rm arc}}$ (Figure \ref{fig:lattice_boot}). 
This result implies that a topological feature whose birth-death pair locates around $(1/2,\sqrt{2}/2)$ shows the difference between the $\cP^{0.1\sqrt{3}}$ and $\cQ^{0.1}$.
\begin{figure}[htbp]
\begin{center}
\caption{The red region is $\mathrm{Box}(x;r)$ (left). For each point $z$ in the index set $I_{s}^{t}(x,r)$ in $\mathrm{Box}(x;r)$ (center), we compare the confidence intervals for $\phi_{*}\cP^{0.1\sqrt{3}} f^{k,w}_{z}$ and $\phi_{*}\cQ^{0.1} f^{k,w}_{z}$.
This result uses $w_{0}$ and shows that there is no intersection between the confidence intervals on $z \in I_{6}^{17}(x,r)$.}
\label{fig:lattice_boot}
\includegraphics[width=1\textwidth]{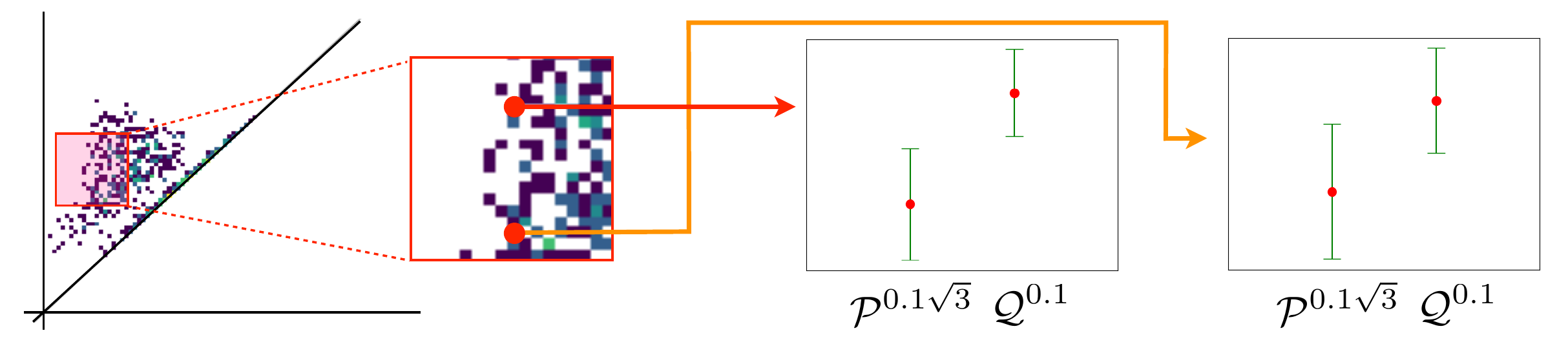}
\end{center}
\end{figure}

\subsection{Mat\'ern Hard-Core Point Process}
\label{subsec:matern}
The Mat\'ern hard-core point process \cite{Ma60} generates a random point sets, which is related to a mathematical model of a granular packing system in material science.
Let $\lambda>0$.
When $N$ is drawn from the Poisson distribution with intensity $\lambda$, the (random) point set 
\[
\{(a_{i},b_{i}) \mid i=1,\ldots,N, ~ a_{1},\ldots,a_{N},b_{1},\ldots, b_{N}, \iid \sim \mathrm{Unif} ([0,1])\}
\]
is denoted by $X^{\lambda}_{0}$ and the underlying probability distribution of $X^{\lambda}_{0}$ is denoted by $\cP^{\lambda}_{0}$, which is known as the {\em homgeneous Poisson point process} in $[0,1]^{2}$ with intensity $\lambda$.
Since $X^{\lambda}_{0} \sim \cP^{\lambda}_{0}$ is randomly scattered in $[0,1]^{2}$, the pairwise distance of some two points in $X^{\lambda}_{0}$ may be less than some value $R>0$.
\cite{Ma60} proposes two thinning rules so that all pairwise distances of any remaining two points are grater than or equal to $R$.

In the Type I way, a point $x_{i} \in X^{\lambda}_{0}$ is removed if there exists another point $x_{j} \in X^{\lambda}_{0}$ such that $\norm{x_{i}-x_{j}} < R$.
Then, the resulting point set is denoted by 
\[
X^{\lambda,R}_{1}:=\{x_{i} \in X^{\lambda}_{0} \mid \norm{x_{i}-x_{j}} > R \mbox{ for any } j \neq i\}
\]
and the corresponding probability distribution is denoted by $\cP^{\lambda,R}_{1}$.

In the Type II way, for $w_{1},\ldots,w_{N}, \iid \sim \mathrm{Unif}([0, 1])$, we use $w_{i}$ as a label of $x_{i} \in X^{\lambda}_{0}$.
While the Type I way always thins $x_{i}$ if the $R$-ball centered at $x_{i}$ contains another $x_{j}$, the Type II way does not thin $x_{i}$ if the label $w_{i}$ is greater than the other $w_{j}$ whose corresponding $x_{j}$ are in the $R$-ball centered at $x_{i}$.
Then, the resulting point set is denoted by 
\begin{align*}
X^{\lambda,R}_{2}:=\{x_{i} \in X^{\lambda}_{0}  \mid w_{i} > w_{j} \mbox{ for any } w_{j} \mbox{ such that } \norm{x_{i}-x_{j}} \leq R \}
\end{align*}
and the corresponding probability distribution is denoted by $\cP^{\lambda,R}_{2}$.
If $w_{1}>w_{2}$ and $\norm{x_{1}-x_{2}} \leq R$, we have $x_{1} \in X^{\lambda,R}_{2}$ and $x_{2} \not\in X^{\lambda,R}_{2}$, while both $x_{1}$ and $x_{2}$ are not in $X^{\lambda,R}_{1}$.
Thus, the Type II way remains more points than the Type I way.
\begin{figure}[htbp]
\begin{center}
\caption{From left to right, $X^{\lambda}_{0}, X^{\lambda,R}_{1}, X^{\lambda,R}_{2}$ (top) and the corresponding persistence diagrams in dimension $1$ (bottom).
In a red encircled region in the point sets, for two points whose distance is less than $R$, we can see that the Type I way removes both two points and the Type II way remains at least one points.} 
\label{fig:matern}
\includegraphics[width=0.95\textwidth]{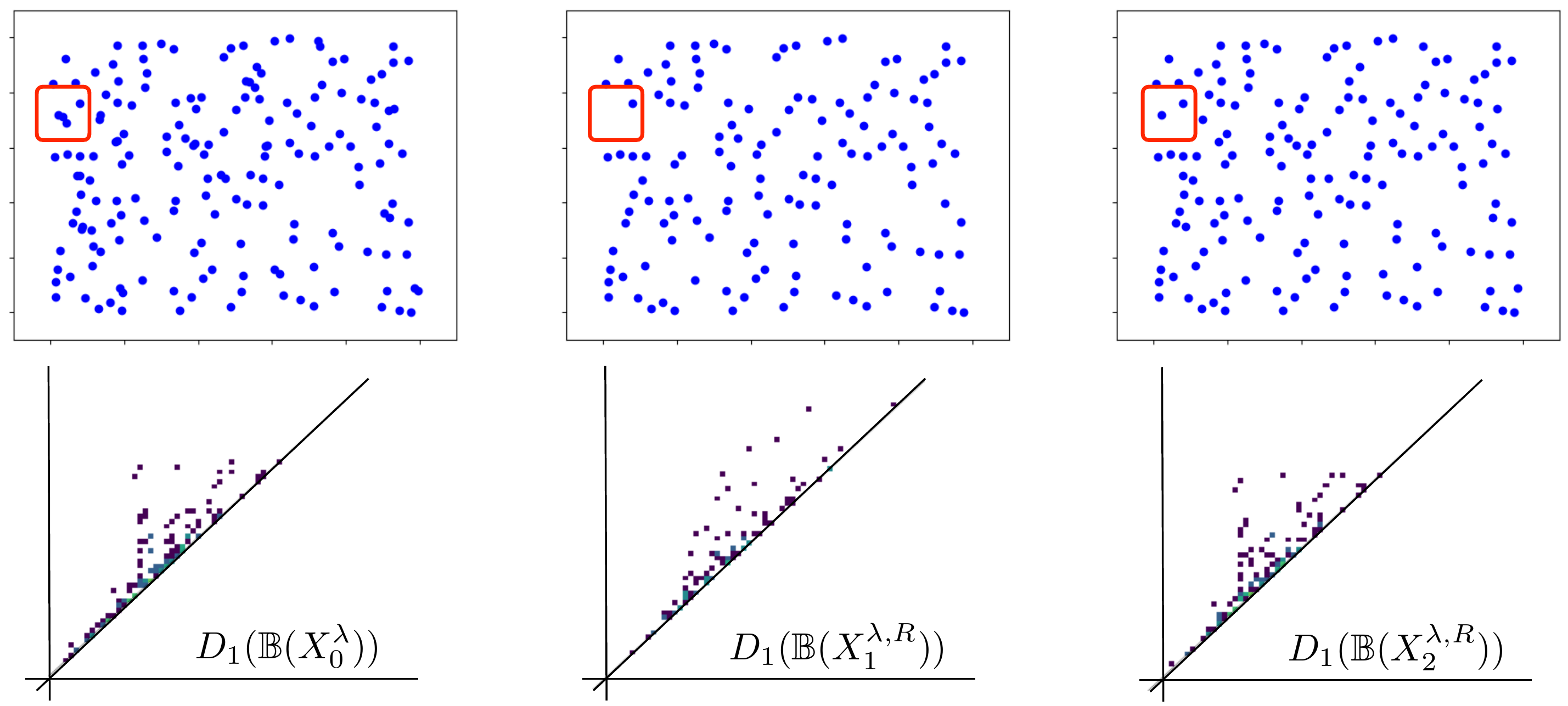}
\end{center}
\end{figure}
In Figure \ref{fig:matern}, we can see that the Type I way removes many small clusters of points in $X^{\lambda}_{0}$ and the Type II way remains at least one point from these clusters in $X^{\lambda}_{0}$, which implies that blank regions (large topological features) in $X^{\lambda}_{0}$ and $X^{\lambda,R}_{2}$ are similar and the differences between $X^{\lambda}_{0}$ and $X^{\lambda,R}_{2}$ are seen in small clusters of points (small topological features).
We fix $\lambda=200$ and $R=0.03$, take $n=50$ samples from each distributions $\cP^{\lambda}_{0}, \cP^{\lambda,R}_{1}, \cP^{\lambda,R}_{2}$, and fix the parameters in Algorithm \ref{alg:type_error_pd} by $\alpha=0.01, q=1, m=40, N=1000$.

\begin{table}[htbp]
\begin{center}
\caption{Each value in this table is the type I error (left) and the type II error (right).}
\label{table:test_matern}
\begin{tabular}{|c|c||c|c|c|} \hline
\multicolumn{2}{|c||}{The null hypothesis}     & $\cP^{\lambda}_{0}=\cP^{\lambda,R}_{1}$ & $\cP^{\lambda}_{0}=\cP^{\lambda,R}_{2}$   & $\cP^{\lambda,R}_{1}=\cP^{\lambda,R}_{2}$   \\ \hline\hline
		     		& $w_{{\rm 0}}$           &0.00/0.00				&0.001/0.00  			&0.00/0.00 \\ \cline{2-5}
PWK		       	    	& $w_{{\rm 1}}$           &0.00/0.00				&0.00/0.75   			&0.00/0.00 \\ \cline{2-5}
    				& $w_{{\rm arc}}$        &0.00/0.00				&0.00/0.91   			&0.00/0.00 \\ \cline{1-5}
\multicolumn{2}{|c||}{PSSK}			  &0.00/0.00				&0.00/0.84 			&0.00/0.00    \\ \hline     
\multicolumn{2}{|c||}{universal PSSK}	  &0.00/0.00				&0.00/0.96 			&0.00/0.00    \\ \hline     
\multicolumn{2}{|c||}{Persistence landscape}&0.00/0.00				&0.00/0.99 			&0.00/0.01    \\ \hline     
\multicolumn{2}{|c||}{Sliced Wasserstein}    &0.00/0.00				&0.00/0.98			&0.00/0.00    \\ \hline    
\end{tabular}
\end{center}
\end{table}

Table \ref{table:test_matern} shows that the PWK vector with $w_{0}$ only correctly classifies $\cP^{\lambda}_{0}$ and $\cP^{\lambda,R}_{2}$.
This is an interesting result because most of the research in statistical TDA regard a generator with large persistence as an important feature and a generator with small persistence as a noisy feature, which is a common view of a persistence diagram in TDA community, but $w_{0}$ treat them equivalently.
Since the Type II way remains relatively large topological features and changes relatively small topological features by its thinning rule (Figure \ref{fig:matern}), it is considered that generators with small persistence cause the difference between two persistence diagrams of $\cP^{\lambda}_{0}$ and $\cP^{\lambda,R}_{2}$.
To see this, we focus on a point close to the diagonal set and compare the confidence intervals where the index set is set as $I_{1}^{20}(x,r)$ for $x=(0.03, 0.04)$ and $r=0.01$ (Figure \ref{fig:matern_boot}).
Then, the confidence intervals for $\phi_{*}\cP^{\lambda}_{0} f^{k,w}_{z}$ and $\phi_{*}\cP^{\lambda,R}_{2}f^{k,w}_{z}$ have intersections for on $z \in I_{0}^{20}(x,r)$ for $w_{1}$ and $z \in I_{14}^{19}(x,r)$ for $w_{{\rm arc}}$.
On the other hand, the confidence intervals by $w_{0}$ do not have an intersection on any $z \in I_{0}^{20}(x,r)$.
We remark that, when we select $x=(0.05, 0.07)$ whose persistence is lager than $(0.03,0.04)$, all confidence intervals by $w_{0},w_{1}, w_{{\rm arc}}$ have intersections.
These results imply a topological feature whose persistence is relatively small will describe the difference between $\cP^{\lambda}_{0}$ and $\cP^{\lambda,R}_{2}$.

\begin{figure}[htbp]
\begin{center}
\caption{The red region is $\mathrm{Box}((0.03,0.04);0.01)$ (left).
For $z=(0.03,0.04)$, the confidence intervals for $\phi_{*}\cP^{\lambda}_{0} f^{k,w}_{z}$ and $\phi_{*}\cP^{\lambda,R}_{2}f^{k,w}_{z}$ are shown (right).}
\label{fig:matern_boot}
\includegraphics[width=1\textwidth]{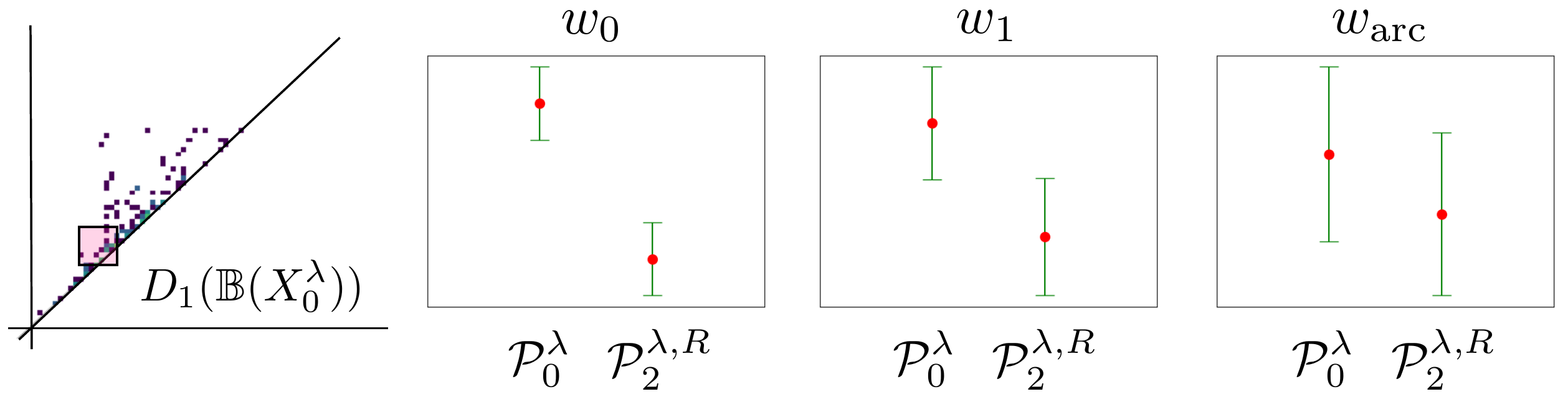}
\end{center}
\end{figure}

\section{Conclusion}
\label{sec:conclusion}
In this paper, for the PWK vector, we have shown the strong law of large numbers, the central limit theorem, and the stability of the expectation and constructed a uniform confidence band for the expectation.
Since the PWK vector is composed of a positive definite kernel $k$ and a weight function $w$, we also have characterized classes of $k$ and $w$, which are needed to satisfy the above theorems.
In numerical experiments, we have discussed the role of a weight function in the two sample test and shown the advantage of the PWK vector over other statistical methods.
Results in Section \ref{subsec:matern} give us an interesting counterexample to the common view in TDA that points in a persistence diagram far from the diagonal are always seen as important features.

\subsection*{acknowledgements}
This work is supported by JSPS Research Fellow (17J02401).

\appendix
\section{Kernel two sample test}
\label{sec:ktst}
In this section, we briefly review the kernel two sample test, following \cite{GBRSS07,GBRSS12}.
\subsubsection*{Gereneral framework of the two sample test}
Let $(\cX, \cB_{\cX})$ be a topological space, $P$ and $Q$ be probability distributions on $\cX$, $X_{1},\ldots,X_{m}, \iid \sim P$, $Y_{1},\ldots, Y_{n}, \iid \sim Q$, and $\theta(\bm{X}_{m}, \bm{Y}_{n})$ be a statistics.
If $H_{0}$ is true, the statistics $\theta(\bm{X}_{m}, \bm{Y}_{n})$ depends only on $P$ because $Y_{1},\ldots, Y_{n}, \iid \sim Q=P$.
Here, we assume that the upper $\alpha$-quantile $\hat{\xi}_{m,n,\alpha}$ which satisfies $\prob(\theta(\bm{X}_{m}, \bm{Y}_{n}) \leq \hat{\xi}_{m,n,\alpha})=1-\alpha$ is computable when $H_{0}$ is true.
When $\alpha$ is set a small value, if a pair of realizations $(\bm{x}_{m}, \bm{y}_{n})$ of $(\bm{X}_{m},\bm{Y}_{n})$ satisfies $\hat{\xi}_{m,n, 1-\alpha/2} \leq \theta(\bm{x}_{m}, \bm{y}_{n}) \leq \hat{\xi}_{m,n,\alpha/2}$, we conclude that the hypothesis $H_{0}$ is {\em accepted} under $H_{0}$.
Otherwise, we conclude that $H_{0}$ is {\em rejected}.
The threshold $\alpha$ is called the {\em significance level} and $\alpha=0.01$ or $0.05$ is often used.

\subsubsection*{Kernel method for the two sample problem}
Let $k$ be a measurable positive definite kernel on $\cX$ satisfying $\int_{\cX}\int_{\cX} k(x,y)^{2}dP(x)dQ(y) < \infty$.
In \cite{GBRSS07,GBRSS12}, a statistics
\begin{align}
&\mathrm{MMD}_{u}(\bm{X}_{m},\bm{Y}_{n};k)^{2} \nonumber \\
&:= \frac{1}{m(m-1)}\sum_{i =1}^{m} \sum_{j \neq i}^{m} k(X_{i},X_{j}) + \frac{1}{n(n-1)}\sum_{a=1}^{n}\sum_{b \neq a}^{n} k(Y_{a},Y_{b}) - \frac{2}{mn}\sum_{i=1}^{m} \sum_{a=1}^{n}k(X_{i},Y_{a}) \label{eq:mmd_u}
\end{align}
is used to the two sample test and the distribution function is given as follows:
\begin{theorem}[Theorem 8 in \cite{GBRSS07}, Theorem 12 in \cite{GBRSS12}]
\label{thm:wchi_dist}
Under the null hypothesis $H_{0}$, $n\mathrm{MMD}_{u}(\bm{X}_{n},\bm{Y}_{n};k)^{2} \convd \sum_{i=1}^{\infty}\lambda_{i}(z_{i}^{2}-2)$ where $z_{1},\ldots, \iid \sim \cN(0,2)$, $\{\lambda_{i}\}_{i=1}^{\infty}$ are the solutions to the eigenvalue equation
\begin{align}
\int_{\cX} \tilde{k}(x, x')\psi_{i}(x)dP(x) = \lambda_{i} \psi_{i}(x'), \label{eq:mmd_eigen}
\end{align}
and $\tilde{k}(x_{i},x_{j})=k(x_{i},x_{j})-\int_{\cX}k(x_{i},x)dP(x)-\int_{\cX}k(x,x_{j})dP(x)-\int_{\cX}\int_{\cX}k(x,x')dP(x)dP(x')$.
\end{theorem}

In order to obtain the distribution of $\sum_{i=1}^{\infty}\lambda_{i}(z_{i}^{2}-2)$ numerically, we approximate the eigenvalues $\{\lambda_{i}\}_{i=1}^{\infty}$ in Equation \eqref{eq:mmd_eigen}.
Let $\tilde{\bm{k}}$ denote the centered Gram matrix of $\{x_{1},\ldots,x_{n}\}$ whose $(i,j)$ component is given by $(\tilde{\bm{k}})_{i,j}=k(x_{i},x_{j})- n^{-1} \sum_{b=1}^{n}k(x_{i},x_{b})-n^{-1} \sum_{a=1}^{n}k(x_{a},x_{j}) + n^{-2} \sum_{a,b=1}^{n}k(x_{a},x_{b})$ and $\{\hat{\mu}_{i}\}_{i=1}^{n}$ be the set of the eigenvalues of $\tilde{\bm{k}}$.
Then, it is shown from Theorem 1 in \cite{GFHS09} that $\sum_{i=1}^{n}\hat{\lambda}_{i}(z_{i}^{2}-2) \convd \sum_{i=1}^{\infty}\lambda_{i}(z_{i}^{2}-2)$ where $\hat{\lambda}_{i}=n^{-1}\hat{\mu}_{i}$.
Therefore, the upper $\alpha$-quantile of $n\mathrm{MMD}_{u}(\bm{X}_{n}, \bm{Y}_{n})^{2}$ is numerically obtained from the histogram of $\sum_{i=1}^{n}\hat{\lambda}_{i}(z_{i}^{2}-2)$.
Since the current null hypothesis is $P=Q$, we have $Y_{i} \sim P$ and we can approximate the eigenvalues on the aggregated data, that is, the eigenvalues are approximated by the centered Gram matrix of $\{x_{1},\ldots,x_{n},y_{1},\ldots,y_{n}\}$.
We estimate the quantile of $\sum_{i=1}^{2n}\hat{\lambda}_{i}(z_{i}^{2}-2)$ by the (standard) bootstrap method.
To sum up, the algorithm of the kernel two sample problem is given as follows:
\begin{center}
\begin{minipage}{0.9\linewidth}
\begin{algorithm}[H]
\caption{The algorithm for the type I error by the kernel method}
\label{alg:type_error_kernel}
\begin{algorithmic}[1]
\REQUIRE $\alpha \in (0, 1)$. $m,N \in \lN$. $x_{1},\ldots,x_{n}$: realizations of $X_{1},\ldots,X_{n}, \iid \sim P$. $y_{1},\ldots,y_{n}$: realizations of $Y_{1},\ldots,Y_{n}, \iid \sim Q$. $k$: a real-valued measurable positive definite kernel on $\cX$.
\STATE Compute the upper $\alpha$-quantile $\hat{\xi}_{n,1-\alpha}$ of $n\mathrm{MMD}_{u}(\bm{X}_{n}, \bm{Y}_{n};k)^{2}$ by the bootstrap method on the aggregated data.
\FOR{$\ell= 1,\ldots,N$}
\STATE Resample $m$ samples $x_{1,\ell},\ldots,x_{m,\ell}$ (resp. $y_{1,\ell},\ldots,y_{m,\ell}$) from $\{x_{1},\ldots,x_{n}\}$ (resp. $\{y_{1},\ldots,y_{n}\}$). 
\STATE Compute $m\mathrm{MMD}_{u}(\bm{x}_{m,\ell}, \bm{y}_{m,\ell};k)$.
\ENDFOR
\STATE Compute $\hat{p}:=N^{-1} \sum_{\ell=1}^{N} \one_{(m\mathrm{MMD}_{u}(\bm{x}_{m,\ell}, \bm{y}_{m,\ell};k) \leq \hat{\xi}_{n,1-\alpha})}$.
\RETURN $1-\hat{p}$. 
\end{algorithmic}
\end{algorithm}
\end{minipage}
\end{center}
Since the output $\hat{p}$ of Algorithm \ref{alg:type_error_kernel} is the acceptance ratio of $H_{0}$, $1-\hat{p}$ is the type I error when $H_{0}$ is true.

\bibliographystyle{apalike}
\bibliography{reference}

\end{document}